\documentclass{article}
\usepackage[letterpaper,top=2cm,bottom=2cm,left=3cm,right=3cm,marginparwidth=1.75cm]{geometry}

\usepackage{amsthm, amsmath, amssymb, amsfonts}
\usepackage[utf8]{inputenc}
\usepackage[english]{babel}
\usepackage{url}
\usepackage{mathtools}

\usepackage{tikz}
\usepackage{tikz-3dplot}
\tdplotsetmaincoords{70}{110}
\usepackage{pgfplots}
\usepackage{pgfplotstable}
\usepackage{booktabs}
\usepackage{algorithm}
\usepackage{algpseudocode}
\usepackage{caption}
\usepackage{subcaption}
\usepackage{mathtools}
\mathtoolsset{showonlyrefs}

\newcommand{\NNN}{\nonumber\\}

\numberwithin{figure}{section}
\numberwithin{table}{section}
\numberwithin{equation}{section}

\newenvironment{abstr}[1]{ \vspace{.05in}\footnotesize
	\parindent .2in
	{\upshape\bfseries #1. }\ignorespaces}{\par\vspace{.1in}}
\newenvironment{Abstract}{\begin{abstr}{Abstract}}{\end{abstr}}
\newenvironment{keywords}{\begin{abstr}{Key words}}{\end{abstr}}
\newenvironment{AMS}{\begin{abstr}{AMS subject classifications}}{\end{abstr}}

\newtheorem{theorem}{Theorem}[section]
\newtheorem{lemma}[theorem]{Lemma}
\newtheorem{corollary}[theorem]{Corollary}
\newtheorem{proposition}[theorem]{Proposition}

\theoremstyle{definition}

\newtheorem{remark}[theorem]{Remark}

\allowdisplaybreaks[4]

\allowdisplaybreaks[4]

\begin{document}
	
	\title{Subspace decomposition with defect diffusion coefficient%
	}
	\author{Dilini Kolombage\footnotemark[3]\and Axel M\aa lqvist\footnotemark[2] \and Barbara Verf\"urth\footnotemark[3]}
	\date{}
	\maketitle
	
	\renewcommand{\thefootnote}{\fnsymbol{footnote}}
	\footnotetext[3]{Institute for Numerical Simulation, University of Bonn, Friedrich-Hirzebruch-Allee 7, 53115 Bonn, Germany.}
    \footnotetext[2]{Department of mathematical sciences, Chalmers University of Technology and University of Gothenburg, Chalmers Tvärgata 3, 41296 Göteborg, Sweden.}
	\renewcommand{\thefootnote}{\arabic{footnote}}
	
	\begin{Abstract} 
Elliptic diffusion problems with multiscale heterogeneous coefficients lead to poorly conditioned discrete systems and therefore require effective preconditioning strategies. While subspace decomposition preconditioners perform well for fixed realizations of the coefficient, their repeated construction becomes prohibitively expensive in uncertainty quantification settings, particularly in Monte-Carlo simulations, where a large number of fine-scale realizations must be treated. In this paper, we propose an offline-online approximation of a subspace decomposition preconditioner that exploits the localized structure of the random defects.  The preconditioner is constructed from local subspace solves that are precomputed offline for a small set of reference configurations and efficiently combined online for arbitrary realizations. We analyze the spectral properties of the resulting offline-online approximation operator and confirm its robustness and efficiency through numerical experiments. 
	\end{Abstract}
	
	

\section{Introduction}
\label{sec:1}
Elliptic diffusion problems with heterogeneous coefficients arise in a wide range of applications including composite materials, porous media and other environmental modeling. Due to the strong heterogeneity and possible high contrast of these coefficients, often characterized by multiple scales, the corresponding discrete problems typically lead to large, ill-conditioned linear systems, for which iterative solvers require suitable preconditioning to achieve robust convergence. Domain decomposition and subspace correction methods are standard tools in this regard, as they provide effective preconditioners based on a global coarse correction combined with localized subdomain solves; see, for example, the foundational subspace correction framework of Xu in \cite{MR1193013}, and the classical domain decomposition literature \cite{MR2104179, MR1857663}. In this study, we are primarily interested in the two-level additive Schwarz framework consisting of a geometric coarse space and overlapping localized subdomain corrections adapted for the multiscale setup; see \cite{MR3536998}. For a fixed coefficient realization, such realization-dependent preconditioning strategies---and additive Schwarz methods in particular---are well understood and can be applied effectively in combination with standard finite element discretizations. 

The computational situation changes drastically when many realizations of the diffusion coefficients must be treated, as in Monte Carlo (MC) sampling and uncertainty quantification. In this setting, each realization gives rise to a different discrete operator and, in general, to a different realization-dependent preconditioner. Although the action of an already constructed additive Schwarz preconditioner can be evaluated efficiently through independent coarse and patch-local solves, repeatedly preparing the local solvers for every realization may become prohibitively expensive, since the setup of domain decomposition methods involves many fine-scale local factorizations, inversions, or highly accurate inner iterative solves. This bottleneck motivates amortized strategies, in which expensive local computations are performed offline and then reused across many solves. 

Existing work addresses this challenge through different reuse mechanisms. Closest in spirit to the present work are from the class of sampling based methods that combine domain decomposition with a two-phase split. In this direction, \cite{MR3828864} constructs local polynomial-chaos (PC) approximations of subdomain contributions to build realization-dependent Schur-complement operators, while in \cite{MR4247004}, the authors utilize local surrogates of subdomain operators to construct realization-dependent preconditioners. Related ideas have also been developed for stochastic BDDC methods, where local BDDC components are approximated offline using subdomain-local random parametrizations and evaluated online for each realization \cite{tu2025stochastic}. These methods substantially reduce the realization-dependent setup cost while retaining the local and coarse structure of the underlying domain decomposition method. Their offline (phase I) approximations are, however, typically constructed as functions of continuous local random variables obtained from Karhunen--Lo\`eve (KL) or PC representations. In contrast, the proposed method here exploits a finite collection of recurring local coefficient configurations generated by spatially localized defects. 

Several further strategies have been proposed for accelerating sequences of parameter-dependent or sampled linear systems. In \cite{venkovic2024preconditioners}, the authors introduce a finite library of offline preconditioners by quantizing the random coefficients via a truncated KL representation and Voronoi cells, and use the nearest centroid coefficient to choose which preconditioner to apply. Nearby preconditioning methods introduced in \cite[Chapter 4]{MRPhD} solve the stochastic Helmholtz equation and reuse a preconditioner constructed for one realization to accelerate the solve for another realization whenever the coefficients are sufficiently close in an appropriate norm. Other approaches interpolate sampled inverse operators \cite{MR3484401} or update a reference preconditioner through sparse approximate maps between successive system matrices \cite{MR4273698}. Reduced-order preconditioners combine a conventional fine-level preconditioner with parameter-dependent reduced coarse corrections constructed from solution or error snapshots \cite{MR3782402}, while related methods interpolate these coarse spaces on the Grassmann manifold \cite{MR5029272}. Alternatively, Krylov-subspace recycling retains approximate invariant or search spaces across related systems \cite{MR2272183, MR2541285, MR2350023}, whereas low-rank tensor Krylov methods compress and solve an entire parameter-dependent family simultaneously \cite{MR2854614}. Complementary sampling strategies reduce Monte Carlo cost through variance-reduction and defect-based control-variate techniques for stochastic homogenization \cite{MR3479886,MR3341127}. Although these approaches do not specifically aim at the Monte Carlo coefficient-realizations, they address the many-query difficulty, but predominantly reuse or approximate global preconditioners, global reduced spaces, or global Krylov information. The present method instead constructs each realization-dependent preconditioner from reusable patch-local inverse actions determined directly by the local defect geometry. 

A separate line of work concerns functional stochastic methods including intrusive stochastic Galerkin and spectral stochastic finite element discretizations. These formulations produce one large coupled system involving both spatial and stochastic degrees of freedom and have motivated mean-based and block preconditioners \cite{MR2491431, MR2486837}, Kronecker-structured and truncation preconditioners \cite{MR2740635, MR2639600, MR4331976}, and stochastic domain decomposition methods \cite{MR3129536, MR3242991}. Although these methods also reuse deterministic operator information, their algebraic setting differs from the sequence of independent realization-wise systems considered in this work.

In contrast to these existing approaches where randomness is typically represented through global random fields or low-dimensional parameterizations, this study is motivated by composite materials with a periodic (or engineered) microstructure where manufacturing damage or impurities introduce rare, highly localized defects. In such materials, most periodic cells coincide with a fixed background configuration, while only a small subset contains defects at weakly randomized locations (see Figure \ref{coef_patterns} for a two-dimensional illustration). These weakly random defect models are widely studied in stochastic homogenization and related multiscale numerical methods  \cite{MR2864078, MR3479886, MR3341127, MR4941773, MR4378546, MR2818410}. In particular, we consider coefficients of the form 
\begin{equation}
\label{coeff:1}
A(x,\gamma) =A_\varepsilon(x) + b_{p,\varepsilon}(x,\gamma)\, B_\varepsilon(x)
\end{equation}
where $A_\varepsilon$ denotes the periodic background coefficient, $B_\varepsilon$  the localized defect contribution, and $b_{p,\varepsilon}$ is a random indicator function encoding the presence or absence of defects on the periodic cells with a given probability $p$. The precise mathematical definition is deferred to Section \ref{sec:2}. This local defect structure has an important consequence for the multiscale discretization. While the fine mesh resolves individual defects, the coarse mesh does not, so that interior coarse scale patches differ only through the local defect patterns they contain. Consequently, the number of distinct local defect configurations that may occur on an interior patch remains finite and independent of the total number of Monte Carlo realizations. This finite set of admissible configurations forms the key structural property exploited in the present work.
\begin{figure}[h]
\centering
\includegraphics[width=0.5\textwidth]{ 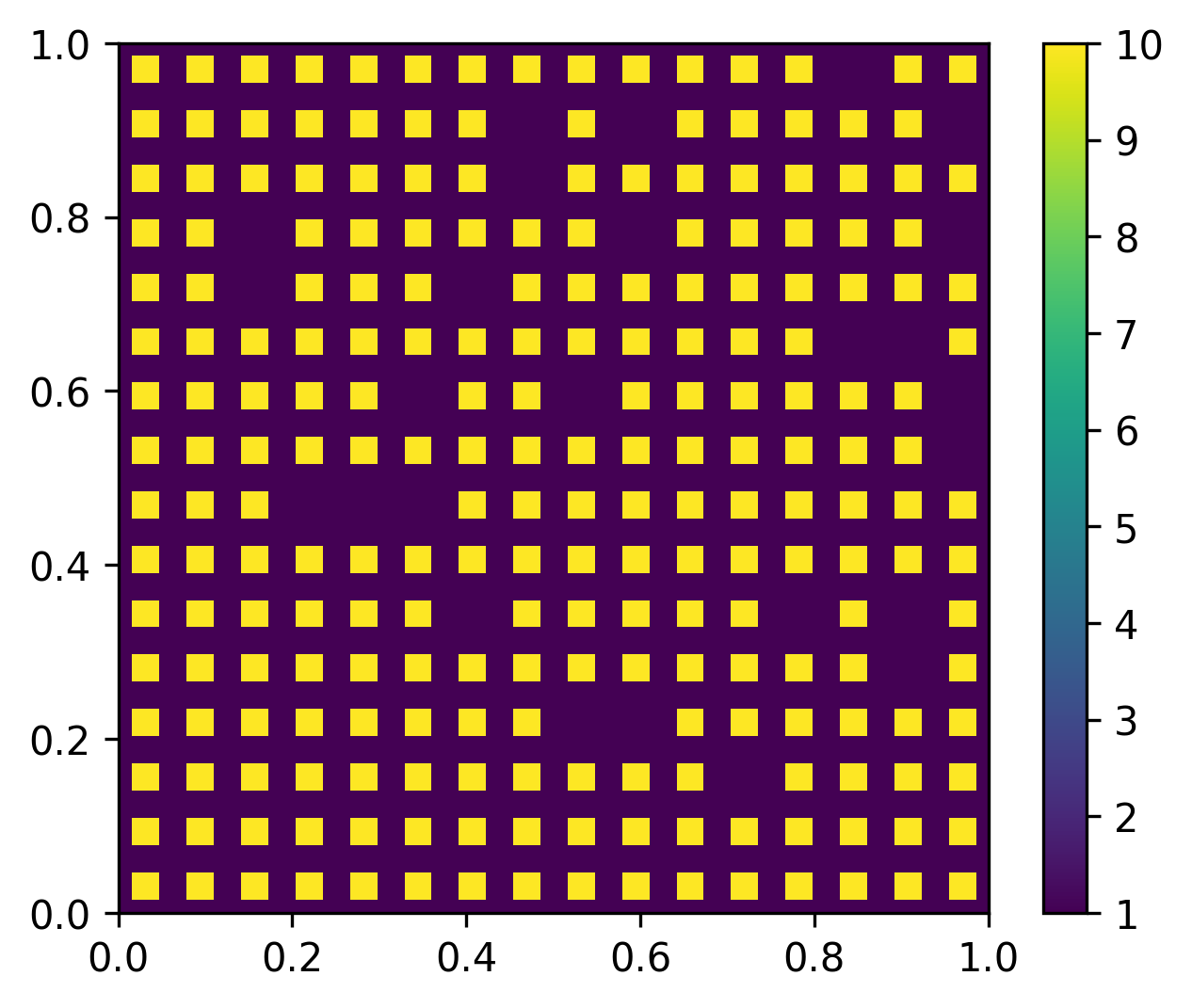}
\caption[0.6\textwidth]{The random erasure model in two dimensions with $\alpha=1 $, $\beta=10, p=0.1,\, \varepsilon=2^{-4}$ and $h=2^{-9} $.}
\label{coef_patterns}
\end{figure}
Building on this observation, we propose a two-level additive Schwarz offline-online preconditioner tailored to defect-type coefficients. In the offline stage, we compute and store patch-local factorizations for these single-defect reference coefficients including the defect-free background. In the online stage, for each realization and each patch, the local preconditioner contributions are assembled purely algebraically as combinations and permutations of these precomputed operators, without additional patch local solves. The coarse component is treated analogously through an element-wise offline-online assembly but solved exactly.

The main contribution of this work is to incorporate local defects into a rigorous and efficient approximation of the two-level additive Schwarz preconditioner. We show that the proposed offline-online construction yields an operator $\widetilde{\mathcal{B}}(A)$ that acts as a perturbation of the additive Schwarz preconditioner $\mathcal{B}(A)$, and derive bounds relating the spectral behavior of $\widetilde{\mathcal{B}}(A)$ and the convergence of the preconditioned conjugate gradient (PCG) method to the offline-online approximation error. Numerically, we demonstrate that $\widetilde{\mathcal{B}}(A)$ performs close to $\mathcal{B}(A)$ across varying defect probabilities and contrasts, while substantially reducing the setup cost in comparison to $\mathcal{B}(A)$, also outperforming the reuse of a single defect-free preconditioner. To assess the performance of our preconditioner across different complexity levels, we additionally test it on several defect geometries.

The remainder of this paper is organized as follows. In Section \ref{sec:2}, we introduce the model problem and the subspace decomposition framework that forms the basis of our offline-online method considered in this work. Section~\ref{sec:3} presents the proposed offline-online strategy for approximating the exact subspace decomposition preconditioner in detail. In Section \ref{sec:4}, we analyze the spectral properties of the preconditioned operator and the convergence results of the preconditioned conjugate gradient method. Implementation aspects and a detailed run-time complexity analysis are discussed in Section \ref{sec:5}. Numerical experiments illustrating the performance of the offline-online method in comparison to several base methods are presented in Section \ref{sec:6}. Finally, Section \ref{sec:7} concludes the paper with a summary and outlook. 

\section{Subspace decomposition framework}
\label{sec:2}
In this section, we introduce the model problem, its finite element discretization, and the subspace decomposition framework, following the classical subspace correction method by Xu \cite{MR1193013}; see also \cite{MR3536998,MR4654621}; which form the basis of the method proposed in this paper. The purpose of this section is to fix notation and recall standard concepts.

\subsection{Model problem}
\label{subsec:2.1}
In the present work, we consider the computational domain to be the (unit) hypercube $\Omega \subset \mathbb{R}^d$, $d \in \{1,2,3\}$ equipped with nested uniform Cartesian meshes. The proposed offline-online construction exploits this structured setting; the corresponding patch geometry is introduced in Section \ref{sec:2.2}. Extensions to more general domains are discussed in Remark \ref{rem:2.1}.
We define 
\begin{equation}
H^1(\Omega) := \{ v \in L^2(\Omega) : \nabla v \in [L^2(\Omega)]^d \},
\end{equation}
where $ V := H_0^1(\Omega)$ denotes the subspace of $H^1(\Omega)$ consisting of functions with vanishing trace on $\partial \Omega$. We consider the diffusion problem: find $u \in V$ such that
\begin{equation}
\label{eq:model-problem}
a(u,v) := \int_\Omega A\nabla u \cdot \nabla v \, dx = \int_\Omega f v \, dx =: (f,v)\quad \forall v \in V
\end{equation}
where $f \in L^2(\Omega)$ and the diffusion coefficient $A \in L^\infty(\Omega)$ satisfies the uniform ellipticity bounds $0 < \alpha \le A \le \beta$ for almost every $x \in \Omega$. The associated energy norm is given by $\|v\|_a := a(v,v)^{1/2}$. Throughout this work, we consider random diffusion coefficients $A(x,\gamma)$, where the above properties hold for each realization $\gamma$. The coefficient is taken to be
\begin{equation}
\label{coeff}
A(x,\gamma) =A_\varepsilon(x) + \sum_{j \in I \subset \mathbb{Z}^d} \chi_{\varepsilon(j+\mathcal Q)}(x)\, \hat b_p^j(\gamma)\, B_\varepsilon(x).
\end{equation}
Here, $A_\varepsilon$ denotes the deterministic periodic background coefficient and $B_\varepsilon$ the deterministic local defect coefficient, defined by  \begin{equation}
A_\varepsilon(x) =A_\mathrm{per}(x/\varepsilon)\,, \quad B_\varepsilon(x) =B_\mathrm{per}(x/\varepsilon)
\end{equation}
where $A_\mathrm{per}$ and $B_\mathrm{per}$ are $1$-periodic functions satisfying
\begin{equation}
 0<\alpha \leq \mathrm{ess inf}\, A_\mathrm{per} \leq \mathrm{ess sup}\, A_\mathrm{per} \leq \beta < \infty 
 \end{equation}
 and
 \begin{equation}
0<\alpha \leq \mathrm{ess inf}\, (A_\mathrm{per}+B_\mathrm{per}) \leq \mathrm{ess sup}\, (A_\mathrm{per}+B_\mathrm{per}) \leq \beta < \infty.
\end{equation}
The set $\mathcal Q\subset[0,1]^d$ denotes the reference defect cell and specifies the support of a single defect within one periodic cell, thereby determining the local defect geometry. Furthermore, $\chi_{\varepsilon(j+\mathcal Q)}$ denotes the characteristic function of the scaled and translated copy $\varepsilon(j+\mathcal Q)$ where 
\begin{equation}
    I=\{j\in\mathbb Z^d:\varepsilon(j+\mathcal Q)\subset\Omega\}.
\end{equation}
The random variables $\{\hat b_p^j\}_{j\in I}$ are independent Bernoulli random variables with parameter $p$, i.e., $\hat b_p^j=0$ with probability $1-p$ and $\hat b_p^j=1$ with $p$. 
We illustrate the choice of parameters in this model, considering, for example, the random erasure coefficient in Figure \ref{coef_patterns}. Here, we choose $\mathcal Q = [0.25,0.75]^d$ with
\begin{equation}
    A_{\mathrm{per}}(y)=\begin{cases}
\beta, & y\in\mathcal Q\\
\alpha, & \text{otherwise}
\end{cases}\,\, \quad \text{ and } \quad
\qquad
B_{\mathrm{per}}(y)=
\begin{cases}
\alpha-\beta, & y\in\mathcal Q\\
0, & \text{otherwise}.
\end{cases}\,
\end{equation}
Therefore, $A_\varepsilon$ represents the periodic background material, while $B_\varepsilon$ encodes the local modification associated with an activated defect. Consequently, whenever $\hat b_p^j=1$, the inclusion in the cell $\varepsilon(j+\mathcal Q)$ is replaced by the background value $\alpha$, thereby creating a local defect.

We discretize $\Omega$ using a fine Cartesian mesh $\mathcal{T}_h$ that resolves all microscale features of the coefficient $A$ and define the conforming finite element space $V_h \subset V$ consisting of continuous, piece-wise bilinear functions, i.e., $Q_1$ elements on the mesh. The corresponding discrete problem reads: find $u_h \in V_h$ such that
\begin{equation}
a(u_h,v) = (f,v) \quad \forall v \in V_h.
\end{equation}
In addition, we introduce a nested coarse mesh $\mathcal{T}_H$ with mesh size $H > h$, and denote by $V_H \subset V_h$ the associated coarse finite element space. Throughout the paper, we assume that $\mathcal{T}_h$ is a refinement of $\mathcal{T}_H$.

\subsection{Patch structure and subspaces}
\label{sec:2.2}
In this section we introduce the patch geometry and the associated overlapping subspaces used in this work. For simplicity, all illustrations are shown in two dimensions, although the construction is valid for $d \leq 3$. Following the classical subspace decomposition framework, the fine scale finite element space $V_h$ is decomposed into overlapping local subspaces associated with vertex-centered patches of the coarse mesh. These local spaces form the building blocks of the additive Schwarz preconditioner introduced in Section \ref{subsec:2.3}.

To this end, let $\{\varphi_i\}_{i \in \mathcal{N}}$ denote the nodal basis of the coarse space $V_H$, where $\mathcal{N}$ is the index set of coarse interior vertices. For each $i \in \mathcal{N}$, we define the corresponding subdomain, (hereafter also referred to as a patch)
\begin{equation} \omega_i := \operatorname{supp}(\varphi_i), 
\end{equation}
that is, the union of all coarse elements sharing the vertex associated with $\varphi_i$. Using these patches, we define a family of local subspaces
\begin{equation} V_i := V_h \cap H^1_0(\omega_i), \quad i \in \mathcal{N} 
\end{equation}
and set $V_0 := V_H$. The collection $\{V_i\}_{i=1}^{|\mathcal{N}|}$ forms the overlapping decomposition of $V_h$ underlying the additive Schwarz method. Further observe that the patches $\omega_i$ satisfy the uniformly bounded overlap property; that is, almost every point in $\Omega$ is contained in at most $2^d$ patches independently of $h$ and $H$. Figure \ref{fig:patches} illustrates the resulting patch geometry on a uniform Cartesian mesh.
 	\begin{figure}[H]
 		\centering
 		\begin{tikzpicture}[scale=1/19*4]
 			\def\Size{26}
 			\def\step{0.1}
 			\def\scalebasis{0.478}
 			\def\opabasis{1}
 			\pgfmathtruncatemacro\size{(\Size - 1) / 2 - 1}
 			\foreach \i in {0,...,\size}{
 				\foreach \j in {0,...,\size}{
 					\fill[black] (2*\i+1, 2*\j+1) rectangle ++(1, 1);
 				}
 			} 
 			\fill[white, opacity=0.8, draw=black] (0.5, 0.5) rectangle ++(\Size  - 2, \Size - 2);
 			\draw[step=4cm, gray, shift={(0.5cm, 0.5cm)}] (0.02, 0.02) grid (\Size - 2.02,\Size - 2.02);
 			\draw[thick] (0.5, 0.5) rectangle ++(\Size  - 2, \Size - 2);
 			 			
 			\draw[red, very thick] (4.5,12.5) rectangle ++(8,8);
 			\node[red, above] at (8.5,20.5) {$\omega_i$}; 
 			\draw[purple, very thick] (8.5,4.5) rectangle ++(8,8);
 			\node[purple, below] at (12.5,4.5) {$\omega_j$};
 			\node[gray, above] at (5, \Size - 1.5) {$\mathcal{T}_H$};
 			\node[above] at (\Size - 2.5, \Size - 1.5) {$\Omega$};
 		\end{tikzpicture}
 		\caption{Possible pattern of a coefficient $A_\varepsilon$ and the patches $\omega_i$.}
 		\label{fig:patches}
 	\end{figure}
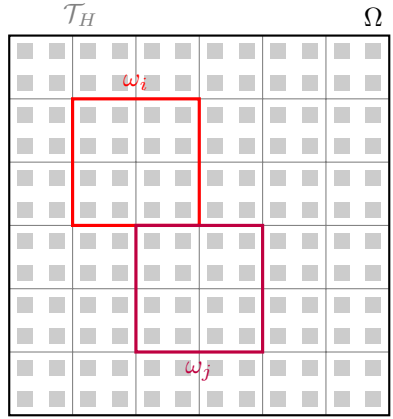
Because the coarse mesh is uniform and Cartesian, all interior patches are identical up to translation. This translation invariance plays a central role in the offline-online construction developed in Section \ref{sec:3}, where local operators computed on a single reference patch are reused throughout the computational domain.

\begin{remark}
\label{rem:2.1}
For more general Lipschitz domains than unit hypercubes, the patches as defined above may no longer be identical up to translation (depending on the mesh structure). Consequently, the translation-invariant property exploited in the offline-online strategy may be weakened. However, this can be addressed by introducing a few additional reference configurations or treating the affected patches separately, still retaining the overall efficiency of the offline-online approach if the domain is not too complicated.
\end{remark}

\subsection{Local projections and additive preconditioner}
\label{subsec:2.3}
We now introduce the classical two-level additive Schwarz preconditioner following the subspace correction framework introduced in \cite{MR1193013}. From this point on, all operators and spaces are understood on the discrete finite element spaces $V_h$ and $V_H$, introduced in Section \ref{subsec:2.1}. The purpose of this section is to establish the exact preconditioner that will subsequently be approximated by the proposed offline-online construction.
Recall that $(\cdot,\cdot)$ denotes the $L^2(\Omega)$ inner product. We let $V_h'$ be the dual space of $V_h$ equipped with the duality pairing $\langle \cdot,\cdot \rangle : V_h' \times V_h \to \mathbb{R} $. Then every $f \in L^2(\Omega)$ is identified with an element of $V_h'$ through
\begin{equation}
\langle f, v \rangle := (f,v) \quad \forall v \in V_h.
\end{equation}
The discrete diffusion operator $\mathcal{L}(A): V_h \to V'_h$ is defined by
\begin{equation}
\langle \mathcal L(A) v, w \rangle := a(v,w) \quad \forall v,w \in V_h.
\end{equation}
For each subspace $V_i$ for $i =1,\dots,|\mathcal{N}| $, we define the corresponding local solution operator $ \mathcal B_i(A) : V_h' \to V_i$ by
\begin{equation}
\label{eq:local-inverse}
a(\mathcal B_i(A) g, w) = \langle g, w \rangle \quad \forall w \in V_i
\end{equation}
for a given $g \in V'_h$. Our aim is to apply $\mathcal B_i(A)$ to the functional of the form $g=\mathcal L(A) v$ yielding the associated local Ritz projection $\mathcal P_i(A) : V_h \to V_i$ defined by
\begin{equation}
    \label{eq:Pexact}
\mathcal P_i(A) = \mathcal B_i(A) \circ \mathcal L(A),
\end{equation} which is equivalent to the variational characterization
\begin{equation}
\label{eq:ritz-projection}
a(\mathcal P_i(A) v, w) = a(v,w) \quad \forall w \in V_i.
\end{equation}
The exact two-level additive Schwarz preconditioner is therefore given by 
\begin{equation}
\label{eq:B_prec}
    \mathcal B(A) := \sum_{i=0}^{|\mathcal{N}|} \mathcal B_i(A)
\end{equation}
where the corresponding preconditioned operator is 
\begin{equation}
\label{eq:P(A)}
    \mathcal P(A) := \sum_{i=0}^{|\mathcal{N}|}\mathcal P_i(A) = \mathcal B(A) \circ \mathcal L(A).
\end{equation}
Hereafter, whenever the dependence on the coefficient is clear from the context, we omit it for simplicity and write $ \mathcal L:=  \mathcal L(A),\, \mathcal B_i:=  \mathcal B_i(A),\, \mathcal B:=  \mathcal B(A)$ and $ \mathcal P :=\mathcal P(A)$. The index $i=0$ in \eqref{eq:P(A)} corresponds to the global coarse space contribution. The associated operator $\mathcal B_0: V_h' \to V_H$ is defined analogously to the local operators $\mathcal B_{i\geq1}$ by solving a variational problem on the coarse space $V_0 := V_H$. It represents the global, coarse-scale correction in the subspace decomposition, while the operators $\mathcal B_{i \ge 1}$, provide localized fine-scale corrections supported on the overlapping patches $\omega_i$.

For completeness and to facilitate comparison with the standard additive Schwarz literature, e.g., \cite{MR2373954}, we also state the equivalent algebraic formulation corresponding to \eqref{eq:B_prec}, which is used in the numerical implementation.  Recall that applying the local solution operator $\mathcal B_i$ amounts to solving a local finite element problem on the patch $\omega_i$.  We, therefore let $R_i$ denote the restriction operator from the global finite element space to the degrees of freedom associated with the local subspace $V_i$, and let $ K_i:=K_i(A)$ denote the corresponding local stiffness matrix on $\omega_i$. Then the matrix representation of the exact two-level additive Schwarz preconditioner can be written as
\begin{equation}
\label{eq:B_matrix}
B :=R^{\top}_0(K_0)^{-1}R_0+\sum\limits_{i=1}^{N}R^\top_i(K_i)^{-1}R_i.
\end{equation}
Throughout the remainder of the paper, we employ the operator formulation for the theoretical analysis, while the numerical implementation is based on the equivalent matrix formulation in \eqref{eq:B_matrix}. 

In practice, the discrete system is solved using the preconditioned conjugate gradient (PCG) method with $\mathcal B$ as a left preconditioner. The operator $\mathcal P = \mathcal B \circ \mathcal L$  is introduced primarily for the subsequent analysis, since its spectral properties determine the convergence behavior of the PCG iteration.  In the ideal case, the local operators $\mathcal B_i$ and the coarse operator $\mathcal B_0$ are computed with respect to the true coefficient $A(x, \gamma)$. However, in the many-realization setting, the repeated construction of these local operators becomes computationally expensive. We address this issue through the proposed offline-online strategy introduced in the next section.

\section{Offline-online approximation of the preconditioner}
\label{sec:3}
We now introduce an offline-online approximation to the exact additive Schwarz preconditioner introduced in Section \ref{sec:2}. The proposed approach retains the exact coarse correction while replacing the realization-dependent local operators by inexpensive approximations assembled from a finite collection of precomputed reference operators. This approach naturally separates into two stages: an offline phase performed once, and an online phase executed for each realization.

We assume that the coarse mesh $\mathcal{T}_H$ is aligned with the periodicity of $A_\varepsilon$ in the sense that the restriction $A_\varepsilon|_T$ is identical for all coarse elements $T \in \mathcal{T}_H$. Furthermore, we assume a scale separation of the form
\begin{equation}
h \le \varepsilon < H,
\end{equation}
where $\varepsilon$ is not asymptotically small with respect to $H$. This regime is typical for defect problems, in which the microscale is resolved on the fine mesh but does not average out at the coarse-scale. We further recall from Section \ref{sec:2.2} that all patches share the same geometry up to translation. Combined with the localized nature of the defect coefficient, this implies that each local problem depends only on the restriction of the coefficient to its corresponding patch. This enables us to precompute the local quantities on a single reference patch as explained in the offline phase below.

\subsection{Offline phase}

The aim of the offline phase is to precompute, once and for all, a set of reference local operators that can subsequently be reused in the online phase to approximate the realization-dependent operators corresponding to arbitrary realizations of the diffusion coefficient.
Because all interior patches $\omega_i$ share the same geometry  (see Figure \ref{fig:patches}), the  restriction of the coefficient $A(\cdot,\gamma)$ to $\omega_i$ is fully determined by the local defect configuration within the patch. This allows local problems posed on different interior patches to be mapped, by translation, to a single reference patch. Furthermore, since the defect term in \eqref{coeff} is localized, the restriction $A(\cdot,\gamma)|_{\omega_i}$ can be represented as a sum of contributions associated with the individual defect locations contained in that patch. We therefore introduce a finite family of reference coefficients defined on a reference patch $\widehat{\omega}$. The first reference coefficient corresponds to the defect-free background,
\begin{equation}
A^{(0)} := A_\varepsilon|_{\widehat{\omega}}.
\end{equation}
For each admissible defect location intersecting the patch, we define its reference coefficient
\begin{equation}
\label{ref:coef}
A^{(\ell)} := A_\varepsilon|_{\widehat{\omega}} + \chi_{\varepsilon(j_\ell + Q)} B_\varepsilon, \quad \ell = 1,\dots,\texttt{N}_{\mathrm{ref}}    
\end{equation}
where $j_\ell$ enumerates the possible defect positions within $\widehat{\omega}$. In this way, the set $\{A^{(\ell)}\}_{\ell=0}^{\texttt{N}_{\mathrm{ref}}}$ exhaustively represents all admissible single-defect configurations on the reference patch. Because arbitrary local defect configurations are assembled online through linear combinations of these reference contributions (see Sections \ref{sec:3.2} and \ref{sec:5.2}), it is sufficient to construct and store operators only for $A^{(0)}$ and $A^{(\ell \ge 1)}$.

We now define the corresponding reference local solution operators (factorizations)  $\widehat {\mathcal{B}}^{(\ell)}:=  \widehat{\mathcal{B}}(A^{(\ell)})$  for $\ell = 0,\dots,\texttt{N}_{\mathrm{ref}}$ by solving a local variational problem on the reference patch $\widehat{\omega}$. Let $\widehat V \subset H_0^1(\widehat{\omega})$ denote the finite element space on $\widehat{\omega}$ . For a given functional $g \in \widehat{ V}'$, the reference local solution operator
\begin{equation}
\widehat {\mathcal{B}}^{(\ell)} : \widehat V' \to \widehat V
\end{equation}
is defined by
\begin{equation}
\label{eq:ref-local-operator}
\int_{\widehat{\omega}} A^{(\ell)} \nabla \bigl(\widehat {\mathcal{B}}^{(\ell)} g\bigr) \cdot \nabla w \, dx = \langle g, w \rangle \quad \forall w \in \widehat V.
\end{equation}
These operators are now stored for subsequent use in the online phase.  

\subsection{Online phase}
\label{sec:3.2}
Given a realization of $A(\cdot,\gamma)$, the aim of the online phase is to approximate the realization-dependent local operators using only the reference operators computed in the offline phase. Since the defect contributions are localized and additive, we can write
\begin{equation}
\label{eq:local}
A|_{\omega_i}=\sum_{\ell=0}^{\texttt{N}_{\mathrm{ref}}} \mu_\ell^{(i)}A^{(\ell)}
\end{equation}
using the reference coefficients from \eqref{ref:coef}. Here, the coefficients $\mu_\ell^{(i)} \in \{0,1\}$ for  $\ell \ge 1$ indicate whether a defect is present at location $\ell$ within the patch $\omega_i$. Since the defects are non-overlapping and enter additively, the background weight satisfies $\mu_0^{(i)} = 1 - \texttt{N}_{\mathrm{def}}^{(i)}$ where $\texttt{N}_{\mathrm{def}}^{(i)}$ denotes the number of defects intersecting $\omega_i$. In particular
\begin{equation}
\label{eq:mu}
\sum_{\ell=0}^{\texttt{N}_{\mathrm{ref}}} \mu_\ell^{(i)} = 1.
\end{equation}
Since the Bernoulli variables defining the defect locations are known for each realization, the coefficients $\mu^{(i)}_\ell$ can be determined directly from the local configurations by identifying which admissible defect locations are present within the patch $\omega_i$. This requires only a simple local inspection of the patch and incurs no additional computational cost. Further details regarding the determination and implementation of $\mu_\ell^{(i)}$ can be found in \cite{MR4378546,MR4941773}. 

Using the local coefficient representation in \eqref{eq:local}, the exact realization-dependent local operator $\mathcal B_i$ is approximated by an algebraic combination of the precomputed reference local operators
\begin{equation}
\widetilde {\mathcal{B}}_i := \widetilde {\mathcal{B}}_i(A) = \sum_{\ell=0}^{\texttt{N}_{\mathrm{ref}}} \mu_\ell^{(i)}\, \mathcal B_i^{(\ell)}
\end{equation}
where $\mathcal B_i^{(\ell)} := \mathcal B_i(A^{(\ell)})$ is the translated copy of the reference operator $\widehat{\mathcal{B}}^{(\ell)}$ to the physical patch $\omega_i$ acting on the local space $V_i$. Consequently, no local boundary problems are solved during the online phase. For each reference coefficient $A^{(\ell)}$, the associated reference local projection operator is $\mathcal P_i^{(\ell)} := \mathcal B_i^{(\ell)} \circ \mathcal L$ yielding the approximate local projection operator
\begin{equation}
\widetilde{\mathcal P}_i := \widetilde{\mathcal B}_i \circ \mathcal{L}.
\end{equation}
Now, summing over all patches, we obtain the full offline-online preconditioner
\begin{equation}
\label{eq:B_oper}
\widetilde {\mathcal B} := \mathcal B_0 + \sum_{i \ge 1} \widetilde {\mathcal B}_i\,.
\end{equation}
The coarse operator $\mathcal B_0$ acts on the coarse space $V_H$ and is assembled and applied exactly. Since $\dim V_H \ll \dim V_h$, its computational cost is negligible compared to the patch-local solves.  An efficient implementation of  $\mathcal B_0$  is described in Section \ref{sec:5.1}.  The operator $\widetilde {\mathcal B}$ therefore preserves the structure of the exact two-level additive Schwarz preconditioner, while avoiding the repeated construction of realization-dependent local operators.

Analogously to \eqref{eq:B_matrix}, we also state the corresponding matrix realization of \eqref{eq:B_oper} by replacing each exact local inverse by its offline-online approximation. Let $\widehat{K}^{(\ell)}:=\widehat{K}(A^{(\ell)})$ denote the stiffness matrix assembled on the reference patch $\widehat{\omega}$ using the reference coefficient $A^{(\ell)}$. Then the corresponding matrix realization of the approximate offline-online preconditioner becomes
\begin{equation}
\label{eq:Btilde_matrix}
     R^{\top}_0K_0^{-1}R_0 + \sum\limits_{i=1}^{N}R^\top_i\Bigg(\sum\limits^{\texttt{N}_\mathrm{ref}}_{\ell=0}\mu^{(i)}_\ell \big(\widehat K^{(\ell)}\big)^{-1}\Bigg) R_i
\end{equation}
where $ \widetilde{B}_i: = R^\top_i\Big(\sum\limits^{\texttt{N}_\mathrm{ref}}_{\ell=0}\mu^{(i)}_\ell \big(\widehat K^{(\ell)}\big)^{-1}\Big) R_i$ is the matrix corresponding to the approximate local operator $\widetilde{\mathcal B}_i$.

\begin{remark}
    It is possible to generalize the construction of the offline-online preconditioner to also include neighboring pairs of defects in the offline phase. While this would slightly increase the cost of the offline phase it would decrease the number of iterations needed for higher values of p, where neighboring pairs of defects occur more frequently.
\end{remark}
In the next section we analyze the approximation of $\mathcal B$ by  $\widetilde {\mathcal B}$ by deriving perturbation estimates that quantify the approximation error and its effect on the spectrum of the preconditioned operator.

\section{Convergence analysis}
\label{sec:4}
In this section, we discuss the effect of the offline-online approximation on the convergence of the PCG method using classical results from additive Schwarz theory. All operators and spaces used in this section are understood as defined in Section \ref{sec:2}. In particular, we recall that $\mathcal L: V_h \to V'_h$ denotes the stiffness operator associated with the bilinear form $a(\cdot, \cdot)$, while $\mathcal B$ and $\widetilde {\mathcal B}$ denote the exact and offline-online additive Schwarz preconditioners, respectively. Throughout this section, we further fix a realization $\gamma$ of the diffusion coefficient $A(\cdot, \gamma)$; accordingly, all the associated bilinear forms and operators are understood to be associated with the same coefficient.

\subsection{Iteration by subspace decomposition}
We first recall the spectral properties of the exact subspace decomposition preconditioner together with the corresponding convergence results for the PCG method, cf.~\cite{MR3536998,MR4654621}. We further recall that
\[\mathcal B := \sum_{i=0}^{|\mathcal{N}|} \mathcal B_i, \quad \quad \mathcal P:= \mathcal B \circ \mathcal L\]
where $\mathcal B_0$ is the coarse correction, while $\mathcal B_{i\geq 1}$ are the exact local operators introduced in Section \ref{subsec:2.3}.

\begin{lemma}
\label{lem:4.1}
For each $i=0,\dots,|\mathcal N|$, the solution operator $\mathcal B_i$ satisfies the property
\begin{equation}
\label{prop:B}
\langle g_1, \mathcal B_ig_2 \rangle = \langle g_2, \mathcal B_ig_1 \rangle \quad \, g_1, g_2 \in V'_h.
\end{equation}
\end{lemma}

\begin{proof}
    By the definition of $\mathcal B_i$, we have $\langle g_1, \mathcal B_ig_2 \rangle = a(\mathcal B_ig_1, \mathcal B_ig_2)$. Since the bilinear form   $a(\cdot ,\cdot)$ is symmetric,  $a(\mathcal B_ig_1, \mathcal B_ig_2)=a(\mathcal B_ig_2, \mathcal B_i g_1)=\langle g_2, \mathcal B_ig_1\rangle$, where the last equality again follows from the definition of $\mathcal B_i$.
\end{proof}

The next lemma follows immediately from Lemma \ref{lem:4.1} together with the symmetry of the bilinear form $a(\cdot ,\cdot)$, and its proof is therefore omitted.

\begin{lemma}[Symmetry of the local projections]\label{lem:symmideal}
For each $i=0,\dots,|\mathcal N|$, the operator $\mathcal P_i = \mathcal B_i\circ \mathcal L$ is symmetric with respect to the energy inner product, i.e.,
\[
a(\mathcal P_iv,w) = a(v, \mathcal P_iw) \quad \forall v, w \in V_h.
\]    
\end{lemma}
\noindent Note that Lemma \ref{lem:symmideal} can alternatively be proved directly using the variational characterization \eqref{eq:ritz-projection} of $\mathcal P_i$.

Besides the symmetry, positive-definiteness of the operator $\mathcal P$ follows from the stable decomposition and bounded-overlap properties. We adapt these properties from \cite[Sections 3 and 4]{MR3536998} to our Cartesian $Q_1$ finite elements spaces and state the central result in Proposition \ref{prop: Lem3.1}. That is, any $v \in V_h$
admits the decomposition $v = \sum_{i=0}^{|\mathcal N|}v_i$ with $v_0 \in V_H$ and $v_{i\ge1} \in V_i$ satisfying 
\begin{equation}
\label{eq:decomposition}
      \sum_{i=0}^{|\mathcal N|}\|v_i\|_a^2    \leq c\frac{\beta}{\alpha}\|v\|_a^2 \qquad \text{and} \qquad   \|v\|_a^2      \leq (2^d +1) \sum_{i=0}^{|\mathcal N|}\|v_i\|_a^2 
\end{equation}
where the constant $c>0$ depends only on the spatial dimension and the shape regularity of the mesh. Here, $(\beta/\alpha)$ is the coefficient contrast and $(2^d +1)$ bounds the number of active subspace contributions coming from the at most $2^d$ patch contributions and the one coarse space contribution in Cartesian $Q_1$ decomposition.

\begin{proposition}[see Lemma 3.1 in \cite{MR3536998}]
\label{prop: Lem3.1}
    The operator $\mathcal P= \sum_{i=0}^{|\mathcal N|} \mathcal P_i$ is symmetric. Under the stable decompositions and bounded overlap estimates in \eqref{eq:decomposition}, there exist constants $C_1, C_2 > 0$, such that 
\begin{equation}
\label{eq:exact-spectral-bounds}
    C_1 \,\|v\|_{a}^2\, \le\, a( \mathcal Pv, v)\,\le \, C_2 \,\|v\|_{a}^2 \quad \forall v\in V_h
\end{equation}
where one may take $C_1= \alpha/(c\beta)$ and $C_2= 2^d+1$.
\end{proposition}

Consequently, the spectrum of $\mathcal P$ is uniformly bounded, i.e., the spectrum is contained in $[C_1,\; C_2]$, and its condition number is bounded by $C_2/C_1$. If $u_h^{(k)}$ denotes the $k$-th iterate of the preconditioned conjugate gradient method applied to $\mathcal Lu=f$ with a symmetric, positive-definite preconditioner $\mathcal B$, then the classical PCG theory yields the error estimate
\begin{equation}
\label{eq:P-cvg}
\|u_h-u_h^{(k)}\|_a \le
2\left(\frac{\sqrt{\kappa}-1}{\sqrt{\kappa}+1}\right)^k
\|u_h-u_h^{(0)}\|_a,
\end{equation}
where $\kappa$ denotes the condition number of $\mathcal P$; see e.g., \cite{MR2104179}. 

\subsection{Offline-online preconditioner and perturbation framework}
We now analyze the offline-online preconditioner $\widetilde {\mathcal B}$ introduced in Section \ref{sec:3} and its associated preconditioned operator $\widetilde {\mathcal P} := \widetilde {\mathcal B}\circ \mathcal L.$ We first verify that $\widetilde {\mathcal P}$ also inherits the symmetry of the exact preconditioned operator $\mathcal P$, and then establish the spectral perturbation bounds to analyze the convergence of the PCG method with $\widetilde{\mathcal B}$ as a left preconditioner.

\begin{proposition}
For each realization of the coefficient $A$, the operator
$\widetilde {\mathcal P}:V_h\to V_h$ is symmetric with respect to the energy inner product $a(\cdot,\cdot)$, that is,
\begin{equation}
a(\widetilde {\mathcal P}v,w)=a(v,\widetilde {\mathcal P}w)
\quad \forall v,w\in V_h.
\end{equation}
\end{proposition}

\begin{proof}
In order to prove the symmetry of $\widetilde{\mathcal P}$, for arbitrary $v,w \in V_h$, we compute
\begin{equation}
a(\widetilde {\mathcal P}v,w) = a(\widetilde {\mathcal B}\mathcal Lv,w) =\langle \mathcal L\widetilde { \mathcal B}\mathcal Lv,w\rangle.
\end{equation}
Since the stiffness operator $\mathcal L$ is defined via a symmetric bilinear form, we have
\begin{equation}
\langle \mathcal L\widetilde{\mathcal B}\mathcal Lv,w\rangle = \langle \mathcal Lw,\widetilde {\mathcal B}\mathcal Lv\rangle.
\end{equation}
Now recall that each local solution operator $\mathcal B_i$ satisfies the property \eqref{prop:B}. Since every approximate local operator $\widetilde {\mathcal B}_{i\ge 1}$ is defined as a finite linear combination of such solution operators associated with the reference coefficients, it also satisfies \eqref{prop:B}. Therefore, summing over all patches and including the coarse operator $\mathcal B_0$, we conclude that the global approximation operator  $\widetilde{\mathcal{B}}$ satisfies \eqref{prop:B} as well.
\begin{equation}
\langle \mathcal Lw,\widetilde {\mathcal B}\mathcal Lv\rangle = \langle \mathcal Lv, \widetilde {\mathcal B}\mathcal Lw\rangle = a(v,\widetilde {\mathcal P}w).
\end{equation}
This proves that $\widetilde {\mathcal P}$ is symmetric with respect to the energy inner product $a(\cdot, \cdot)$.
\end{proof}
We next want to understand how much the offline-online construction deviates $\mathcal{P}$. To quantify the size of this deviation, we introduce the perturbation constant
\begin{equation}
\label{eq:eta-def}
    \eta := \sup_{v\in V_h \setminus \{0\} }\frac{|a( E(A)v,v) |}{\|v\|_{a}^2}.
\end{equation}
where $E(A) := \widetilde {\mathcal P} - \mathcal P$ is the perturbation between the two operators. In general, $\eta$ measures the size of the offline-online perturbation in the energy inner product, and $\eta=0$ implies that all local operators are assembled exactly. By the definition of $\eta$ we immediately have that
\begin{equation}
\label{prop:perturbation-bound}
|a( E(A)v, v) | \,\le \, \eta\, \|v\|_{a}^2.
\end{equation}

\begin{theorem}[Stability of the offline-online preconditioner.]
\label{thm:oo-stability}
Assume that the exact  preconditioned operator $\mathcal P$ satisfies the spectral bounds \eqref{eq:exact-spectral-bounds} and that the perturbation constant $\eta$ defined in \eqref{eq:eta-def} satisfies $\eta < C_1$. Then the approximate preconditioned operator $\widetilde {\mathcal P}$ satisfies
\begin{equation}
(C_1-\eta)\|v\|_a^2 \;\le\; a( \widetilde  {\mathcal P}v,v) \;\le\; (C_2+\eta)\|v\|_a^2 \quad \forall v\in V_h.
\end{equation}
In particular, the spectrum of $\widetilde {\mathcal P}$ is contained in the interval $[C_1-\eta,\, C_2+\eta]$. Consequently, the condition number of the preconditioned operator $\widetilde {\mathcal P}=\widetilde {\mathcal B} \circ \mathcal L$ satisfies
\begin{equation}
\kappa(\widetilde {\mathcal P}) \;\le\; \frac{C_2+\eta}{C_1-\eta}.
\end{equation}
\end{theorem}

\begin{proof}
By the definition of the perturbation operator $E(A)$, we have,
\begin{equation}
    \widetilde {\mathcal P} = \mathcal P + E(A).
\end{equation}
Therefore, for any $v\in V_h$, 
\begin{equation*}
    \label{eq:split}
    a( \widetilde {\mathcal P}v,v) = a( \mathcal Pv,v) + a( E(A)v,v ).
\end{equation*}
From the spectral bounds for $\mathcal P(A)$, cf. \eqref{eq:exact-spectral-bounds}, and the bound of the perturbation term \eqref{prop:perturbation-bound}, we obtain 
\begin{equation}
a( \widetilde {\mathcal P}v,v )
\ge C_1 \|v\|_a^2 - \eta \|v\|_a^2 =(C_1-\eta)\|v\|_a^2, \quad \text{and}\quad  a( \widetilde {\mathcal P}v,v) \le C_2 \|v\|_a^2 + \eta \|v\|_a^2 = (C_2+\eta)\|v\|_a^2.
\end{equation}
The assumption $\eta<C_1$ ensures that the lower bound is strictly positive, and hence $\widetilde {\mathcal P}$ is uniformly positive definite with respect to the energy inner product. Finally, by the symmetry and positive definiteness of $\widetilde {\mathcal P}$, we can conclude that its spectrum is contained in the interval $[C_1-\eta,\, C_2+\eta]$. Consequently, the condition number of $\widetilde {\mathcal P}$ satisfies
\begin{equation}
\kappa(\widetilde {\mathcal P}) \le \frac{C_2+\eta}{C_1-\eta}. 
\end{equation}
\end{proof}

\begin{remark}
    While positive definiteness of the approximate local operator is not  immediate from the construction alone, Theorem \ref{thm:oo-stability} shows that the resulting global preconditioned operator remains positive-definite provided the perturbation is sufficiently small, precisely if $\eta<C_1$. Moreover, throughout the numerical experiments presented in Section \ref{sec:6}  the offline-online preconditioner $\widetilde{\mathcal B}$  was always found to be positive definite.
\end{remark}

Finally, note that $\widetilde {\mathcal P}= \widetilde {\mathcal B}\circ \mathcal L$ is precisely the operator governing the PCG iteration for the system $\mathcal L(A)u_h = f$. Hence, the spectral bounds established above directly determine the convergence bounds of the PCG method and follows directly from \eqref{eq:P-cvg}.

\begin{corollary}[PCG convergence]
\label{cor:cvg}
Let $u_h \in V_h$ denote the reference solution of $\mathcal Lu=f$, and let $\Tilde{u}_h^{(k)}$ be the $k$-th iterate of the preconditioned conjugate gradient method with preconditioner $\widetilde {\mathcal B}$. Under the assumptions of Theorem \ref{thm:oo-stability}, the iterates satisfy
\begin{equation}
\|u_h-\tilde{u}_h^{(k)}\|_{a} \;\le\; 2\left( \frac{\sqrt{\kappa(\widetilde {\mathcal P})}-1} {\sqrt{\kappa(\widetilde {\mathcal P})}+1} \right)^k \|u_h-\tilde{u}_h^{(0)}\|_{{a}}\,.
\end{equation}
\end{corollary}
Theorem \ref{thm:oo-stability} and Corollary \ref{cor:cvg} show that the offline-online approximation preserves the convergence properties of the exact preconditioner whenever the perturbation constant $\eta$ is sufficiently small. More precisely, since $C_1$ and $C_2$ are independent of the mesh parameters, the PCG convergence rate is robust under mesh refinement, provided that $\eta$ is uniformly bounded above by $C_1$. The remaining question is how the perturbation constant $\eta$ can be estimated in practice. The following proposition derives an a posteriori computable upper bound for $\eta$ in terms of patch-local quantities. This provides a practical criterion for assessing the size of the perturbation within the offline-online framework. 

\begin{proposition}
\label{prop:eta}
Define for $i \in \mathcal{N}$
\begin{equation}
  		E_{i}^2\coloneqq \max_{v\in V_h  \setminus\{0\}}\frac{\|\sum_{\ell=0}^{\texttt{N}_{\mathrm{ref}}}\mu_\ell^{(i)} A^{-1/2}(A-A^{(\ell)})\nabla \mathcal P_i^{(\ell)}v\|_{L^2(\omega_i)}^2}{\|v\|_{a,\omega_i}^2}.
\end{equation}
Then, for any $v \in V_h$ and maximum possible patch overlap $C_\mathrm{overlap}$, it holds that
\begin{equation}
    	\|E(A)v\|_a \leq C_\mathrm{overlap} \Big( \max_{i\in \mathcal{N}}E_i\Big) \|v\|_a.
\end{equation}
Consequently, the perturbation constant $\eta$ satisfies $\eta \leq C_\mathrm{overlap} \Big( \max_{i\in \mathcal{N}}E_i\Big).$
\end{proposition}

\begin{proof}
Recall that $\mathcal B_0$ and therefore also $\mathcal P_0$ is computed exactly for any given $A$. Then we obtain that
    \begin{align}
 			\|E(A)v\|_a^2& =\|(\widetilde {\mathcal P}-\mathcal P)v\|_a^2=(A\nabla(\widetilde {\mathcal P}-\mathcal P)v,\nabla \underbrace{(\widetilde {\mathcal P}-\mathcal P)v}_{=:e})\NNN
            &=\sum_{i\geq 1}(A\nabla(\widetilde{\mathcal P}_i- {\mathcal P}_i)v, \nabla e)_{\omega_i}\NNN
        & \leq \sum_{i\geq 1}\|(\widetilde{\mathcal P}_i-\mathcal P_i)v\|_{a, \omega_i}\|e\|_{a, \omega_i}  \displaybreak[3]\\
        & \leq \bigg(\sum_{i\geq 1}\|(\widetilde{\mathcal P}_i-\mathcal P_i)v\|^2_{a, \omega_i}\bigg)^{1/2}\bigg(\sum_{i\geq 1}\|e\|^2_{a, \omega_i}\bigg)^{1/2}
        \label{eq: PP}
 	\end{align}
where the last inequality is achieved by applying the Cauchy-Schwarz inequality. We hence estimate these terms separately. Note that the goal is to estimate $(\widetilde{\mathcal P}_i-\mathcal P_i)$ in terms of $\|v\|_{a,\omega_i}$, so that with the finite overlap of the patches, the final estimate will be independent from the number of patches. By the definition of $\mathcal P_i$, $\widetilde {\mathcal P}_i$, and $A=\sum_{\ell=0}^{\texttt{N}_{\mathrm{ref}}} \mu_\ell^{(i)} A^{(\ell)}$ on $\omega_i$, we have that
	 \begin{align*}
     	\|(\widetilde{\mathcal P}_i-\mathcal P_i)v\|_{a,\omega_i}^2&=(A\nabla(\widetilde{\mathcal P}_i-\mathcal P_i)v, \nabla\underbrace{(\widetilde{\mathcal P}_i-\mathcal P_i)v}_{=:w\in H^1_0(\omega_i)}) \NNN
	 	&=\sum_{\ell=0}^{\texttt{N}_{\mathrm{ref}}}\mu_\ell^{(i)}(A\nabla \mathcal P_i^{(\ell)}v, \nabla w)- (A\nabla \mathcal P_i v, \nabla w)\NNN
        &=\sum_{\ell=0}^{\texttt{N}_{\mathrm{ref}}}\mu_\ell^{(i)}(A\nabla \mathcal P_i^{(\ell)}v, \nabla w) - \sum_{\ell=0}^{\texttt{N}_{\mathrm{ref}}}\mu_\ell^{(i)}(A^{(\ell)}\nabla \mathcal P_i^{(\ell)}v, \nabla w)\NNN
	 	&=\sum_{\ell=0}^{\texttt{N}_{\mathrm{ref}}}(\mu_\ell^{(i)}(A-A^{(\ell)})\nabla \mathcal P_i^{(\ell)}v, \nabla w)\NNN
        &= \sum_{\ell=0}^{\texttt{N}_{\mathrm{ref}}}(\mu_\ell^{(i)}A^{-1/2}(A-A^{(\ell)})\nabla \mathcal P_i^{(\ell)}v, \, A^{1/2}\nabla w).
        \label{eq: Pi}
	 \end{align*}
Now inserting back for $w$, we get
\begin{align*}
    	\|(\widetilde {\mathcal P}_i-\mathcal P_i)v\|_{a,\omega_i}& \leq 	\|\sum_{\ell=0}^{\texttt{N}_{\mathrm{ref}}}\mu_\ell^{(i)}A^{-1/2}(A-A^{(\ell)})\nabla \mathcal P_i^{(\ell)}v\|_{L^2(\omega_i)}.
\end{align*}
Therefore
\begin{align}
    \sum_{i\geq 1}\|(\widetilde {\mathcal P}_i-\mathcal P_i)v\|^2_{a, \omega_i} \le   \sum_{i\geq 1} E_i^2 \|v\|^2_{a, \omega_i} \leq \Big( \max_{i\in \mathcal{N}}E_i\Big)^2\sum_{i\geq 1} \|v\|^2_{a, \omega_i} 
\end{align}
Using that each $x \in \Omega$  belongs to at most  $C_\mathrm{overlap}$ patches, we deduce that 
\begin{align}
\label{eq:v}
\sum_{i\geq 1}\|v\|^2_{a, \omega_i}  
\leq C_\mathrm{overlap} \|v\|^2_a
\end{align}
Using this also for $\|e\|^2_{a, \omega_i}$, we get with \eqref{eq: Pi} and \eqref{eq:v} in \eqref{eq: PP} that
\begin{equation}
\label{eq:result}
  	\|E(A)v\|_a \leq C_\mathrm{overlap} \Big( \max_{i\in \mathcal{N}}E_i\Big) \|v\|_a.  
\end{equation}  
Finally, by the Cauchy-Schwarz inequality
\begin{equation}
\label{eq:eta}
    |a(E(A)v,v)| \leq \|E(A)v\|_a\|v\|_a.
\end{equation}
Combining \eqref{eq:result} and \eqref{eq:eta}, and taking the supremum over all $v \neq 0$, we obtain
\begin{equation*}
   \eta  \leq C_\mathrm{overlap}\Big( \max_{i\in \mathcal{N}}E_i\Big).\qedhere
\end{equation*} 
\end{proof}
Note that the quantities $E_i$ are patch local indicators and can be computed efficiently within the offline-online strategy using the same approach as in \cite{MR4378546}, see also \cite{MR4119333, MR3945233}. Consequently, Proposition \ref{prop:eta} provides an explicit a posteriori upper bound for the perturbation constant $\eta$. In particular the stability assumption of Theorem \ref{thm:oo-stability} is guaranteed whenever $ C_\mathrm{overlap}\Big( \max_{i\in \mathcal{N}}E_i\Big) < C_1$, thereby providing a practical criterion for verifying the stability of the offline-online preconditioner.

\section{Implementation aspects}
\label{sec:5}

In this section, we describe the practical realization of the offline-online preconditioner introduced in Section \ref{sec:3}. Here we explain the main implementation aspects that enables its efficient assembly and application.  We further discuss the associated computational cost, storage requirements, and implementation remarks that are essential for large-scale Monte Carlo simulations.

\subsection{Practical realization of the preconditioner components }
\label{sec:5.1}
The implementation follows the matrix formulations introduced in \eqref{eq:B_matrix} for the exact preconditioner and \eqref{eq:Btilde_matrix} for the offline-online preconditioner. However, the preconditioner is never assembled as a global matrix. Instead, in all methods considered in this paper, the coarse contribution and all patch-local contributions are represented by $\texttt{LinearOperator}$ objects in Python's \texttt{scipy} package and passed directly to the PCG solver. \newline

\noindent\textbf{Patch-local contributions.} For the reference additive Schwarz preconditioner, the patch-local stiffness matrices are assembled and factorized for every realization, and their action is applied through the corresponding factorizations. In the offline-online preconditioner, by contrast, this realization-dependent local factorizations are replaced by the reuse of precomputed reference matrices through algebraic combinations. More precisely, the local contribution of the exact preconditioner associated with each patch $\omega_i$ is given by
\begin{equation}
     B_i(A) := R_i^\top \left( K_{i} \right)^{-1} R_i 
\end{equation}
where $K_i$ is the patch-local stiffness matrix assembled with coefficient $A$ and homogeneous Dirichlet conditions on $\partial\omega_i$. Thus, constructing the exact preconditioner requires the factorization of one local stiffness matrix for every patch in every realization. The offline-online method avoids repeating these realization-dependent setup by using a small offline dictionary of precomputed reference matrices
\begin{equation}
    \widehat B^{(\ell)} := \bigl(\widehat K^{(\ell)}\bigr)^{-1}, \qquad \ell=0,\ldots,N_{\mathrm{ref}}.
\end{equation}
These reference matrices $\widehat K^{(\ell)}$ are computed once on the reference patch $\widehat{\omega}$ and their factorizations are retained as offline data. No information about a particular realization of the coefficient is required at this stage. We stress that the offline dictionary contains only the defect-free reference operator and one single-defect reference operator for each admissible defect position on $\widehat\omega$, and therefore consists of $N_{\mathrm{ref}}+1$ matrices. Its cardinality is determined solely by the local patch geometry and the ratio $H/\varepsilon$ and is independent of both the number of physical patches and the number of coefficient realizations. The same offline data are therefore reused for every patch and throughout the complete Monte Carlo simulation. The resulting computational and storage implications are discussed in Section \ref{sec:runtime} and Remark \ref{subsec:storage}, and are illustrated quantitatively in Experiment \ref{ex:1}.

In the online phase, for a given realization of the coefficient $A(\cdot,\gamma)$, we loop over all interior patches $\omega_i$. From the defect configuration on $\omega_i$, we determine the weights $\{\mu_\ell^{(i)}\}_{\ell=0}^{\texttt{N}_{\mathrm{ref}}}$ by identifying the active defect cells and matching them with their corresponding positions on the reference patch. The approximate local matrix of the physical patch $\omega_i$ is then assembled algebraically as
\begin{equation}
\label{eq:Bi}
    \widetilde B_{i} :=\mathcal{T}_i \left(  \sum_{\ell=0}^{N_{\mathrm{ref}}} \mu_\ell^{(i)}\widehat B^{(\ell)} \right) \mathcal{T}_i^\top,
\end{equation}
where $\mathcal{T}_i$ is the index permutation that maps the degrees of freedom of the reference patch to those of $\omega_i$. The corresponding contribution to the global preconditioner matrix is 
\begin{equation}
    R_i^\top \widetilde B_i\,R_i.
\end{equation}
Importantly, no fine-scale linear systems are solved in the online phase; only linear combinations and index permutations are required.

\begin{remark}
    Although additive-Schwarz preconditioner can also be implemented in a fully matrix-free manner by replacing the local factorizations with inner iterative solvers, we do not adopt this approach since we assume exact patch-local computations in this study. Furthermore, considering the rather small patch sizes, the cost of these local factorization is rather low.
\end{remark}

\noindent\textbf{Coarse contribution.} We next describe the common realization of the coarse contribution shown in \eqref{eq:B_matrix} as well as \eqref{eq:Btilde_matrix}. That is
\begin{equation}
\label{eq:B0}
    R_0^\top K_0(A)^{-1}R_0
\end{equation}
where $K_0(A)$ is the realization-dependent coarse stiffness matrix. Here $R_0^\top$ denote the canonical prolongation from the coarse space $V_H$  into the fine-scale space $V_h$, represented in the corresponding nodal basis. Since  $\dim V_H \ll  \dim V_h$,  factorizing the coarse matrix is inexpensive compared with constructing the fine-scale patch contributions.  The main question here is assembling $K_0(A)$ efficiently for many realizations of the defect coefficient. This is achieved by an exact offline-online strategy at the coarse-element level. In the offline phase, we fix a reference coarse element $\widehat{T}$ and construct reference coarse-element stiffness matrices $\{K_{0,\widehat{T}}^{(\hat{\ell})}\}_{\hat{\ell}=0}^{\hat{\texttt{N}}_{{\mathrm{ref}}}}$ corresponding to the background coefficient and all single-defect positions $\hat{\ell}= 1 , \dots \hat{\texttt{N}}_{\mathrm{ref}}$. In the online phase, for each coarse element $T\in\mathcal{T}_H$, we determine coefficients $\{\lambda_{\hat{\ell}}^{(T)}\}$ such that
\begin{equation}
 A|_T = \sum_{\hat{\ell}} \lambda_{\hat{\ell}}^{(T)} A^{(\hat{\ell})}_{\widehat{T}}, 
\end{equation}
and assemble
\begin{equation}
\label{eq:lambda}
K_{0,T}(A) = \sum_{\hat{\ell}} \lambda_{\hat{\ell}}^{(T)} K_{0,\widehat{T}}^{(\hat{\ell})}.
\end{equation}
The global coarse stiffness matrix is then obtained by standard finite element assembly and factorized to form the coarse contribution $R_0^\top K_0(A)^{-1}R_0$. Finally the coarse and patch-local contributions are combined as
\begin{equation}
   R_0^\top K_0(A)^{-1}R_0 + \sum_{i=1}^{N}R_i^\top\widetilde B_i(A)R_i.
\end{equation}


\subsection{The algorithm}
\label{sec:5.2}
Algorithm \ref{alg:BP} summarizes the offline-online construction of the additive preconditioner and its use within a standard preconditioned conjugate gradient method. We additionally remark on possible further implementation improvements.

\begin{remark}{(Incremental representation)}
An equivalent representation of the coarse-element matrix in step (17) of Algorithm \ref{alg:BP} is
\begin{equation}
K_{0,T}(A) = K_{0,\widehat{T}}^{(0)} + \sum_{q\in S_T} \left( K_{0,\widehat{T}}^{(q)} - K_{0,\widehat{T} }^{(0)} \right),
\end{equation}
where $S_T$ denotes the set of defect positions within $T$. Both formulations are algebraically identical; the $\lambda$-representation is conceptually simpler, while the incremental form may reduce scaling costs.
\end{remark}

\begin{remark}{(Symmetry reduction)} The offline cost for the coarse-scale matrices $ B_0$ can be further reduced by exploiting geometric symmetries of the coarse element. Defect locations that are symmetry-equivalent lead to stiffness matrices that differ only by permutation of degrees of freedom. Thus, only one representative per symmetry class needs to be stored.
\end{remark}

\begin{remark} In principle, similar symmetry arguments apply to patch-local matrices. However, due to the overlap of patches and the larger local dimensions, the required permutation logic is significantly more complex. For this reason, symmetry reduction is easily applicable only at the coarse level.
\end{remark} 

\begin{algorithm}[h]
\caption{Offline-online construction of the additive preconditioner}
\label{alg:BP}
\begin{algorithmic}[1]
\small
\State \textbf{Input:} Background coefficient $A_\varepsilon$, defect coefficient $B_\varepsilon$, reference cell $Q$, right-hand side $f$
\State Fix $\widehat{\omega} \subset \mathcal{T}_H$ and  $\widehat{T} \subset \widehat{\omega}$.
\Comment{Offline phase}

\State Construct $\{A^{(\ell)}\}_{\ell=0}^{\texttt{N}_{\mathrm{ref}}}$ on $\widehat{\omega}$ and $\{A^{(\hat{\ell})}\}_{\hat{\ell}=0}^{\hat {\texttt{N}}_{\mathrm{ref}}}$ on  $\widehat{T}$.
\For{$\ell = 0,\ldots,\texttt{N}_{\mathrm{ref}}$}
  \State Compute and retain $\widehat B^{(\ell)}$ associated with $A^{(\ell)}$.
\EndFor

\For{$\hat{\ell} = 0,\ldots,\hat {\texttt{N}}_{\mathrm{ref}}$}
  \State Assemble and retain $ K_{0,\widehat{T}}^{(\hat{\ell})}$ associated with $A^{(\hat{\ell})}.$
\EndFor
\Comment{End offline phase}

\For{each $A(\cdot,\gamma)$}
\Comment{Online phase}

  \For{each $\omega_i$}
    \State Determine weights $\{\mu_\ell^{(i)}\}_{\ell=0}^{\texttt{N}_{\mathrm{ref}}}$ according to \eqref{eq:local} and \eqref{eq:mu}.
    \State Construct $\widetilde{B}_i$ as in \eqref{eq:Bi}, and determine $R_i$.
  \EndFor

  \For{each $T \in \mathcal{T}_H$}
    \State Determine weights
    $\{\lambda_{\hat{\ell}}^{(T)}\}_{\hat{\ell}=0}^{\hat{\texttt{N}}_{\mathrm{ref}}}$ and assemble $   K_{0,T}$ as in \eqref{eq:lambda}.
  \EndFor

  \State Compute the coarse contribution from \eqref{eq:B0}.
  \State For a residual vector $r$, use the operator action
  \begin{equation}  R_0^\top K_0(A)^{-1}R_0r + \sum_i R_i^\top\widetilde { B}_iR_i r \end{equation}
  \State to solve $\mathcal L(A)u = f$ by PCG as a linear system.
\EndFor
\end{algorithmic}
\end{algorithm}

\subsection{Run-time complexity}
\label{sec:runtime}

A natural benchmark for any approximate preconditioning strategy is the preconditioner based on the exact coefficient of the realization. This typically shows the best performances but can be too costly when many realizations must be solved. In the rare-defect regime considered here, an equally natural baseline is to construct a single preconditioner from the defect-free background coefficient and reuse it across all samples. When most realizations are close to the background, this approach can be effective while essentially eliminating the setup cost. The proposed offline-online strategy is designed to retain much of the accuracy of the sample-dependent construction while amortizing expensive local computation across samples.

Therefore, in this section, we compare the computational complexity of the proposed offline-online domain decomposition (OO-DD) preconditioner (in Algorithm \ref{alg:BP}) against these two baselines: (i) Direct-DD, i.e., the exact  sample-dependent approach, and (ii) ND-DD, the ``no-defect'' approach, where the preconditioner is constructed once from the defect-free background and reused for all samples. The Direct-DD and OO-DD preconditioners share the same algebraic structure, differing only in the use of exact versus offline-online approximated local patch operators. 

All methods are used as preconditioners in a standard preconditioned conjugate gradient (PCG) solver for the discrete linear system  $\mathcal L(A)u_h = f$. Throughout the discussion, we focus on the cost of constructing and applying the preconditioner. The per-iteration cost of PCG, consisting of one fine-scale matrix-vector product and one application of the preconditioner, is comparable for all three methods and is therefore treated separately.

We assume a scale separation $h \le \varepsilon < H$, where the fine mesh resolves all microscale features of the coefficient. The coarse mesh $\mathcal{T}_H$ is assumed to be aligned with the periodic background coefficient. We consider vertex-centered patches with fixed overlap and layer size. The total number of interior patches satisfies
$
\texttt{N}_{\texttt{P}} \sim H^{-d},
$ up to a constant depending only on the mesh connectivity. In the following, we therefore do not distinguish between the number of coarse elements and the number of patches, as these quantities are of the same order.

For simplicity, we only introduce the following representative costs. $\texttt{T}_{\mathrm{patch}}$: cost of solving one fine-scale local problem on a patch; $\texttt{T}_{\mathrm{comb}}$: cost of forming a linear combination of precomputed patch operators; $\texttt{T}_{\mathrm{PCG}}$: cost of one PCG iteration.
At the theoretical level, all patch-level solves involve matrices of comparable size and structure, so that a single representative cost $\texttt{T}_{\mathrm{patch}}$ is sufficient. Moreover, we denote $\texttt{k}^{\mathrm{D}}$, $\texttt{k}^{\mathrm{ND}}$, and $\texttt{k}^{\mathrm{OO}}$ to be the number of PCG iterations required to reduce the residual below a prescribed stopping tolerance for the Direct-DD, ND-DD, and OO-DD preconditioners, respectively. Throughout the comparison, the tolerance is assumed to be fixed and identical for all methods. We further let $\texttt{N}_{\mathrm{samp}}$ be the number of random samples used.

For the Direct-DD method, the local operators are assembled exactly for each realization of the coefficient. There is no offline phase. For each sample coefficient, one must solve a local fine-scale problem on every patch resulting in a per sample online cost $  \texttt{N}_{\texttt{P}} \, \texttt{T}_{\mathrm{patch}}$. The total cost for $\texttt{N}_{\mathrm{samp}}$ samples is therefore 
\begin{equation}
\texttt{T}^{\mathrm{D}} \;\sim\; \texttt{N}_{\mathrm{samp}} \left( \texttt{N}_{\texttt{P}} \, \texttt{T}_{\mathrm{patch}} + \texttt{k}^{\mathrm{D}} \, \texttt{T}_{\mathrm{PCG}} \right).
\end{equation}

In the ND-DD method, the preconditioner is constructed using only the periodic background coefficient $A_\varepsilon$ and reused unchanged for all samples. Due to periodicity and translation invariance, all interior patches are equivalent. Consequently, the local operators need to be computed only once on a single reference patch and can then be reused by translation. This suggests one time  offline cost $ \texttt{T}_{\mathrm{patch}}$ and zero per sample online cost. The total cost is therefore
\begin{equation}
\texttt{T}^{\mathrm{ND}} \;\sim\; \texttt{T}_{\mathrm{patch}} + \texttt{N}_{\mathrm{samp}} \left( \texttt{k}^{\mathrm{ND}} \, \texttt{T}_{\mathrm{PCG}} \right).
\end{equation}

In the proposed OO-DD method, a small number of reference patch operators corresponding to different single-defect configurations are computed offline. In the online phase, the local operators for a given sample are obtained by linear combinations of these reference operators. No local fine-scale solves are required online. Hence, the one-time offline cost is $\texttt{N}_{\mathrm{ref}} \, \texttt{T}_{\mathrm{patch}}$, while the per-sample online cost is $\texttt{N}_{\texttt{P}} \, \texttt{T}_{\mathrm{comb}}$. Here, $\texttt{N}_{\mathrm{ref}}$ denotes, for simplicity, the total number of reference configurations, including the defect-free background configuration. The total cost is therefore
\begin{equation}
\texttt{T}^{\mathrm{OO}} \;\sim\; \texttt{N}_{\mathrm{ref}} \, \texttt{T}_{\mathrm{patch}} + \texttt{N}_{\mathrm{samp}} \left( \texttt{N}_{\texttt{P}} \, \texttt{T}_{\mathrm{comb}} + \texttt{k}^{\mathrm{OO}} \, \texttt{T}_{\mathrm{PCG}} \right).
 \end{equation}

To understand when the OO-DD is advantageous over the other two methods, one may consider the break-even number of samples $\texttt{N}_{\mathrm{samp}}^\ast$. With respect to the Direct-DD, equating $\texttt{T}^{\mathrm{OO}} = \texttt{T}^{\mathrm{D}}$ yields 
\begin{equation}
\texttt{N}_{\mathrm{samp}}^\ast \approx  \frac{ \texttt{N}_{\mathrm{ref}} \, \texttt{T}_{\mathrm{patch}} }{ \texttt{N}_{\texttt{P}} \bigl( \texttt{T}_{\mathrm{patch}} - \texttt{T}_{\mathrm{comb}} \bigr)
+ \left( \texttt{k}^{\mathrm{D}} - \texttt{k}^{\mathrm{OO}} \right) \texttt{T}_{\mathrm{PCG}} }.
\end{equation}
Since $\texttt{T}_{\mathrm{patch}} \gg \texttt{T}_{\mathrm{comb}}$, the denominator is large and the break-even number of samples is typically small, i.e., the one-time offline-dictionary cost is amortized after a relatively small number of samples.
Similarly, with respect to ND-DD, equating $\texttt{T}^{\mathrm{OO}} = \texttt{T}^{\mathrm{ND}}$ gives
\begin{equation}
\texttt{N}_{\mathrm{samp}}^\ast \approx \frac{ \left( \texttt{N}_{\mathrm{ref}} - 1 \right) \texttt{T}_{\mathrm{patch}} }{ \left(\texttt{k}^{\mathrm{ND}} - \texttt{k}^{\mathrm{OO}} \right) \texttt{T}_{\mathrm{PCG}} - \texttt{N}_{\texttt{P}} \, \texttt{T}_{\mathrm{comb}} }.
\end{equation}
Here, the offline-online method becomes advantageous once the reduction in iteration counts is sufficient to compensate for the additional offline cost. This shows that the offline-online method is getting more and more advantageous as the number of Monte-Carlo simulations in the experiment increases.

\begin{remark}[Asymptotic comparison] We provide a coarse asymptotic comparison of the three methods to highlight their dominant cost contributions.
 We denote by $n_{\mathrm{loc}}$ the dimension of the local finite element space on a single patch, and $N_H := H^{-d}$ the total number of coarse elements. Let $C_{\mathrm{loc}}(n_{\mathrm{loc}})$ and $C_{\mathrm{comb}}(n_{\mathrm{loc}})$ be the cost of solving one local problem exactly, and the cost of forming a linear combination of precomputed local factorizations respectively. In all three methods, the coarse operator is assembled element-wise with linear complexity $\mathcal{O}(N_H)$. Since this cost is of lower order compared to the local patch contributions, it is omitted in the following comparison. In the Direct-DD method, all local problems are solved independently for each realization. The dominant cost therefore scales as \( \mathcal{O}\bigl( \texttt{N}_{\mathrm{samp}} \, N_p \, C_{\mathrm{loc}}(n_{\mathrm{loc}}) \bigr), \) which grows linearly with both the number of samples and patches. In the ND-DD method, all local operators are constructed only once for the background coefficient and reused unchanged for all samples. The setup cost is therefore \( \mathcal{O}\bigl(C_{\mathrm{loc}}(n_{\mathrm{loc}}) \bigr), \) while the per-sample cost arises solely from the PCG iteration with a fixed preconditioner, and scales as \(  \texttt{N}_{\mathrm{samp}} \, \mathcal{O}\bigl( k^{\mathrm{ND}} \, n_h \bigr), \) where $n_h$ is the number of fine-scale degrees of freedom. In practice, this iteration count may increase significantly in the presence of strong defects or high contrast. In the OO-DD method, only $\texttt{N}_{\mathrm{ref}}$ reference patch problems are solved exactly in an offline stage. The corresponding setup cost is therefore \( \mathcal{O}\bigl(
\texttt{N}_{\mathrm{ref}}\, C_{\mathrm{loc}}(n_{\mathrm{loc}}) \bigr).\) In the online stage, the local operators are obtained by recombination of the precomputed reference data. The per-sample cost is thus \(\texttt{N}_{\mathrm{samp}} \, \mathcal{O}\bigl( \texttt{N}_p \, C_{\mathrm{comb}}(n_{\mathrm{loc}}) \bigr),\) where typically $C_{\mathrm{comb}}(n_{\mathrm{loc}}) \ll
C_{\mathrm{loc}}(n_{\mathrm{loc}})$.

If the local patch problems are solved by direct factorization on spaces of dimension $n_{\mathrm{loc}}$, then $C_{\mathrm{loc}}(n_{\mathrm{loc}})=\mathcal{O}(n_{\mathrm{loc}}^3)$. Alternatively, if the local problems are solved approximately using optimal methods, one may assume $C_{\mathrm{loc}}(n_{\mathrm{loc}})=\mathcal{O}(n_{\mathrm{loc}})$. The recombination cost $C_{\mathrm{comb}}(n_{\mathrm{loc}})$ depends on the chosen implementation but is always of lower order. In all cases, the qualitative comparison between the three methods remains unchanged.
\end{remark}

The three preconditioners therefore represent different trade--offs. The Direct-DD method provides the smallest iteration numbers and thus the fastest convergence per sample, but at the price of the largest per-sample setup cost. The ND-DD method is computationally the cheapest in terms of setup, since the preconditioner is reused for all samples, but this advantage is often outweighed by a substantial increase in iteration numbers when defects are present. The OO-DD method balances these two extremes: its offline stage is more expensive than that of the ND-DD, but this cost is averaged over many samples, while its per-sample setup is much cheaper than that of the Direct-DD and its iteration numbers are typically close to those of the Direct-DD. Hence, for large Monte-Carlo simulations, the OO-DD preconditioner provides the most favorable compromise between computational cost and convergence speed, and its one-time offline cost is increasingly amortized as the number of Monte Carlo samples grow. The numerical experiments in Section \ref{sec:6} underline this behavior and illustrate the practical relevance of the offline-online strategy.

\begin{remark} [Storage and parallelization]
\label{subsec:storage}
In the OO-DD method, all offline data are stored on a single reference patch and reused across the domain by translation, permutation, and symmetry arguments. The dominant offline memory  consumption is due to the patch-local offline data associated with the $\texttt{N}_{\mathrm{ref}}$ reference defect configurations. Under an optimal setup, $ \mathrm{Mem}_{\mathrm{off}}  \;=\;  \Theta\!\bigl(\texttt{N}_{\mathrm{ref}}\,n_{\mathrm{loc}}\bigr)$. In addition, a small dictionary of reference coarse element blocks is stored, corresponding to a fixed number of symmetry classes. Since this contribution is independent of $h$, $H$, and the number of samples, this storage memory is asymptotically negligible. Both the offline loop over reference configurations and the online loops over patches and coarse elements can be completely parallelized. 
\end{remark}
\begin{remark}
    A fully matrix-free implementations with an inexact inner solver might become advantageous in situations where we have to work with three-dimensional problems with complex coefficients that are computationally expensive. However, the amortized arguments underlying the OO-DD is independent of how the local solve in Direct-DD is realized. Whether $T_\mathrm{patch}$ is the cost of a sparse factorization on a patch or the cost of a converged inner iterative solver at a high tolerance, the cost given by $T_\mathrm{patch} \times \texttt{N}_\mathrm{samp}$ is still subject to increase along with the number of samples in the Direct-DD computation. In contrast, OO-DD have one-time offline cost of $\texttt{N}_\mathrm{ref}$ independent of the number of samples considered and thereafter, only an online linear combination cost $\texttt{N}_\mathrm{comb}$ that would only change in a negligible amount depending on the defect probability--which is small in the rare defect regime of multiscale materials.
\end{remark}

\section{Numerical Experiments}
\label{sec:6}

In this section, we investigate the practical performance of the proposed OO-DD preconditioner. All experiments are conducted in two spatial dimensions on the unit square $\Omega=[0,1]^2$ with homogeneous Dirichlet boundary conditions. The goal of the experiments is threefold: (i) to assess the accuracy of the OO-DD preconditioner in the presence of random defects, (ii) to compare its efficiency with the Direct-DD and the ND-DD in terms of PCG iteration counts, and (iii) to demonstrate that the observed behavior persists across different defect geometries.

We consider the diffusion problem \eqref{eq:model-problem} discretized using $Q_1$ finite elements on a fine mesh $\mathcal{T}_h$ that resolves all microscale features of the coefficient. A nested coarse mesh $\mathcal{T}_H$ is used to define the subspace decomposition. The meshes are aligned with the periodic background coefficient and satisfy $h \le \varepsilon < H$. In the following experiments, we use the fixed patch size introduced in Section \ref{sec:2.2}, namely vertex-centered patches consisting of one layer of coarse elements. Unless stated otherwise, the right-hand side is chosen as
\begin{equation}
f(x,y)=\sin(\pi x)\sin(\pi y).
\end{equation}
We have additionally verified all implementations for the constant right-hand side $f\equiv 1$, with no qualitative changes in the observed behavior. A direct finite element solution on the fine mesh $\mathcal{T}_h$ is used as a reference solution for all accuracy comparisons.

We compare the proposed OO-DD preconditioner with the two standard preconditioning strategies introduced in Section~\ref{sec:runtime} both within the same subspace-decomposition framework: Direct-DD and ND-DD preconditioners. In the following experiments, we report convergence behavior and average PCG iteration counts for varying defect probability, contrast, and defect geometry. These quantities serve as practical indicators for the robustness and efficiency of the respective preconditioners.

In all cases, the resulting linear systems are solved using the preconditioned conjugate gradient method as implemented in \texttt{scipy.sparse.linalg.cg} package, with a fixed relative residual tolerance of $10^{-6}$ and a per sample maximum iteration count $200$. Unless stated otherwise, we use the discretization parameters $h=2^{-7}$, $\varepsilon=2^{-5}$, and $H=2^{-4}$. All experimental results are averaged over $150$ independent Monte-Carlo realizations of the coefficient.
All experiments discussed in this section are fully reproducible. The corresponding code and data are publicly available at \cite{dilini_kolombage_2026_21769877}.

\subsection{Random erasure model}
\label{ex:1}
We first consider the random erasure model introduced in \eqref{coeff} and depicted in Figure~\ref{coef_patterns}, where defect inclusions of fixed shape are activated independently with probability $p$ as explained in Section \ref{subsec:2.1}. Before discussing the solver performance, we quantify the size of the offline dictionary for the discretization parameters considered here in order to relate the practical setup cost discussed in Sections \ref{sec:5.1} and \ref{sec:runtime}. By definition, each interior patch has side length $2H$ in each dimension. Then for the discretization parameters used here, i.e., $h=2^{-7},\, \varepsilon=2^{-5}$ and $H=2^{-4}$, the reference patch $\widehat{\omega}$ contains $16$ admissible defect locations. Together with the defect-free reference configuration the offline dictionary consists of only $17$ reference local operators, each corresponding to a local problem with $(2H/h -1)^2 = 225$ interior degrees of freedom. Since the computational domain contains $225$ physical patches, the Direct-DD requires $225$ local factorizations for every realization, amounting to $33750$ local factorizations over the $150$ Monte-Carlo samples considered here. In contrast the proposed OO-DD reuses the precomputed $17$ reference factorizations in the offline dictionary for every patch over every sample, corresponding to less than $0.1\%$ of local factorizations required by the Direct-DD over the complete experiment. Moreover, this experiment uses only a moderate sample size of $150$. Since the offline dictionary is constructed only once, whereas Direct-DD factorizations scale linearly with the number of realizations, the relative offline setup cost of the OO-DD decrease further as the number of Monte-Carlo samples increases.

Now we summarize the performance of the three preconditioners for increasing defect probability and contrast in Figures \ref{incl-c} and \ref{incl-p}. Figure \ref{incl-c} reports the average number of PCG iterations as a function of the defect probability $p$ for contrasts $\beta/\alpha = 10, 100$, and $500$ (from left to right). For all tested configurations, the PCG method converges for all three preconditioners across all samples. Among the three methods, the Direct--DD preconditioner consistently yields the smallest number of iterations, as expected, since all local patch operators are assembled exactly for each realization of the coefficient. In contrast, the ND--DD preconditioner exhibits a substantial increase in iteration counts as the defect probability $p$ increases and, in particular, as the contrast $\beta/\alpha$ becomes large. This behavior reflects the growing mismatch between the defect-free preconditioner and the true operator when defects become more frequent and stronger. The vertical bars indicate the statistical variation of the iteration counts after removing outlier samples. While ND-DD exhibits the largest variability across samples, the proposed OO-DD preconditioner shows significantly smaller variations and iteration counts that remain close to those of Direct-DD over the full range of defect probabilities. Although a mild increase in the number of iterations can be observed as $p$ grows, the overall behavior demonstrates that the offline-online approximation successfully captures the dominant local effects induced by the defects. Furthermore, Figure \ref{incl-c} illustrates the contrast dependency of the considered methods at a fixed mesh hierarchy. The similar iteration counts of Direct-DD and OO-DD at each contrast is a strong indicator that the offline-online approximation does not introduce any additional deterioration even at high contrasts. The iteration increase observed in both methods as the contrast grows is therefore likely to be inherited from the underlying two-level Schwarz decomposition and is consistent with the limitations of the coarse level approximations, rather than the offline-online approximation.

\begin{figure}[h]
\includegraphics[width=\textwidth]{ 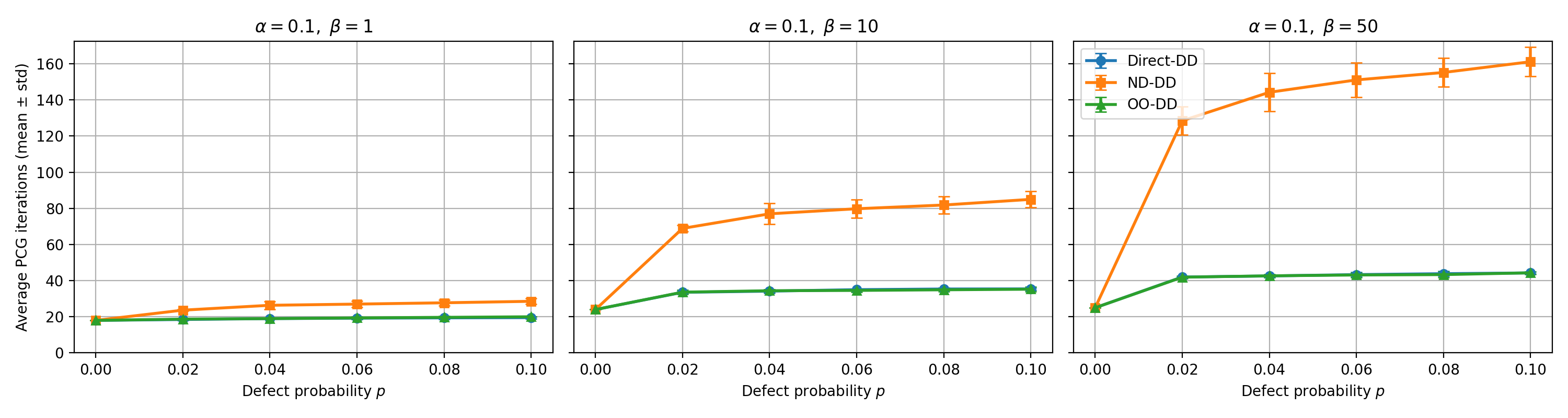}
\caption[0.6\textwidth]{Random erasure model: the average PCG iteration counts against varying defect probabilities for defect contrasts $\beta/\alpha = 10, \, 100$ and $500$ (left to right) respectively.}
\label{incl-c}
\end{figure}

Figure \ref{incl-p} illustrates the convergence histories of the PCG method measured in terms of the root-mean-square error (RMSE) in the energy norm for three given defect probabilities and contrast of $\beta/\alpha = 500$. For all three preconditioners, the energy-norm error decreases steadily with the iteration count, indicating consistent convergence of the method. The convergence curves further highlight clear qualitative differences between the preconditioners. The ND-DD preconditioner exhibits significantly slower decay of the energy error, whereas the proposed OO-DD preconditioner closely follows the convergence behavior of the Direct-DD method across all tested defect probabilities.

\begin{figure}[H]
\includegraphics[width=\textwidth]{ 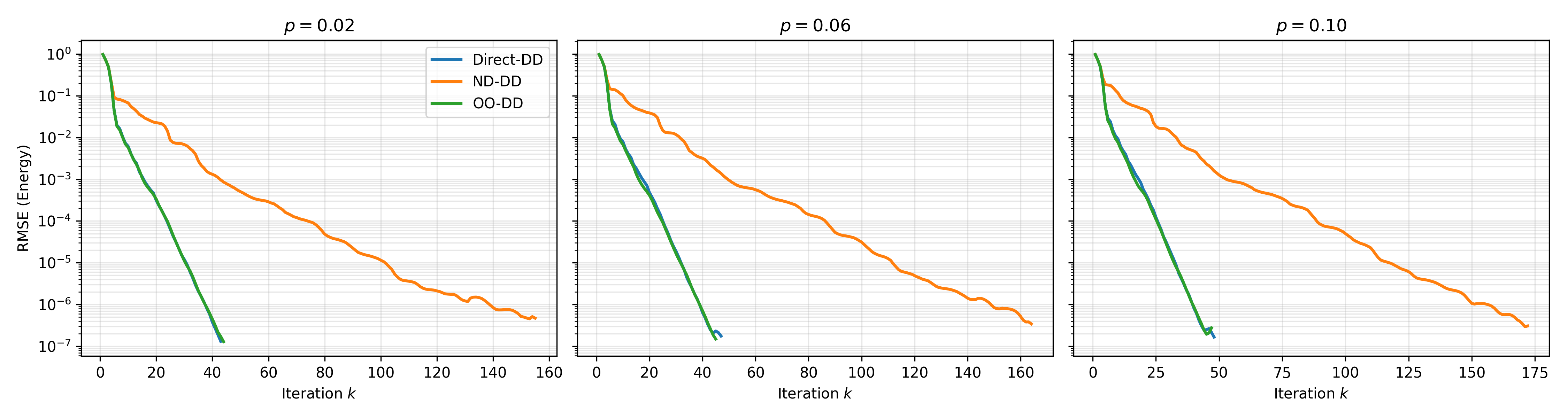}
\caption[0.6\textwidth]{Random erasure model: the RMSE in the energy norm vs iteration count for defect probabilities $p= 0.02, \, 0.06$ and $0.10$ (left to right) respectively at $\beta/\alpha = 500$.}
\label{incl-p}
\end{figure}

We additionally investigate the robustness of the proposed preconditioner under the multiscale mesh refinement in Figure \ref{fig:Hhe_shifted}, where we record the average iteration count over $100$ samples for the random erasure model at a low contrast. In this experiment, we refine the mesh parameters $H,\, \varepsilon$ and $h$ simultaneously while keeping the rations $\varepsilon/H$ and $h/\varepsilon$ fixed. As a result, we obtain the same geometric structure of the local problem allowing us to evaluate the stability of OO-DD method with respect to the uniform refinement of the complete multiscale hierarchy. The slight increase in the iteration count between the first two hierarchy levels (corresponding to $H=2^{-3}$ and $H=2^{-4}$) is attributed to the relatively large discretization at $H=2^{-3}$ for multiscale problems. Beyond this initial increase, the average iteration counts remain nearly constant across all other hierarchy levels suggesting us that the proposed OO-DD preconditioner exhibits good robustness under mesh refinement. In addition to this experiment, we also investigated the dependence of the OO-DD method on the individual mesh parameters by varying $h$ and $H$ respectively, while keeping the remaining two parameters fixed. These experiments exhibited the same behavior and are omitted for brevity. The complete set of experiments are available in the supporting repository \cite{dilini_kolombage_2026_21769877}. We further report that, while offline-online strategy is robust in approximating the Direct-DD even at high coefficient contrasts, Direct-DD has a contrast-dependent convergence behavior as reflected in Proposition \ref{prop: Lem3.1}. 
\begin{figure}[h]
\centering
\includegraphics[width=\textwidth, height=3.5cm]{ 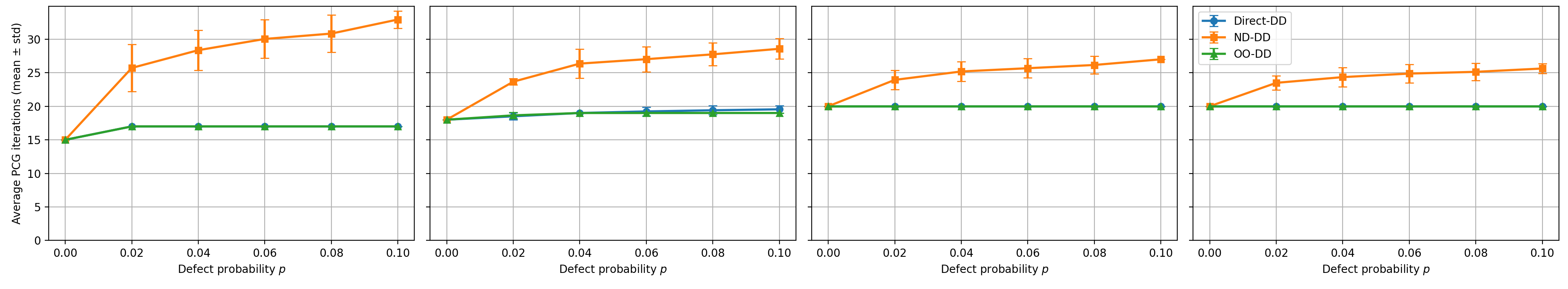}
\caption[0.6\textwidth]{The change of iteration counts for the random erasure model over $100$ samples (with $\beta/\alpha = 10$) at the mesh parameters $(H, \varepsilon, h)= (2^{-3},2^{-4},2^{-6}),$ $(2^{-4},2^{-5},2^{-7}),\, (2^{-5},2^{-6},2^{-8})$ and $(2^{-6},2^{-7},2^{-9})$ respectively from left to right.}
\label{fig:Hhe_shifted}
\end{figure}

\subsection{Other defect models}
In the experiments, apart from the random erasure model depicted in Figure \ref{coef_patterns}, we additionally consider two further random defect models which we introduce as an L-shaped defect model [Figure \ref{incl-models} (left)] and a shifted defect model [Figure \ref{incl-models} (right)]. These models introduce different levels of geometric complexity and, they therefore serve to assess the robustness of the preconditioners with respect to changes in defect shape and spatial localization. Both models still fit into the general framework of \eqref{coeff} modifying the local defect geometry encoded in $B_\varepsilon$.  
\begin{figure}[h]
\includegraphics[width=\textwidth]{ 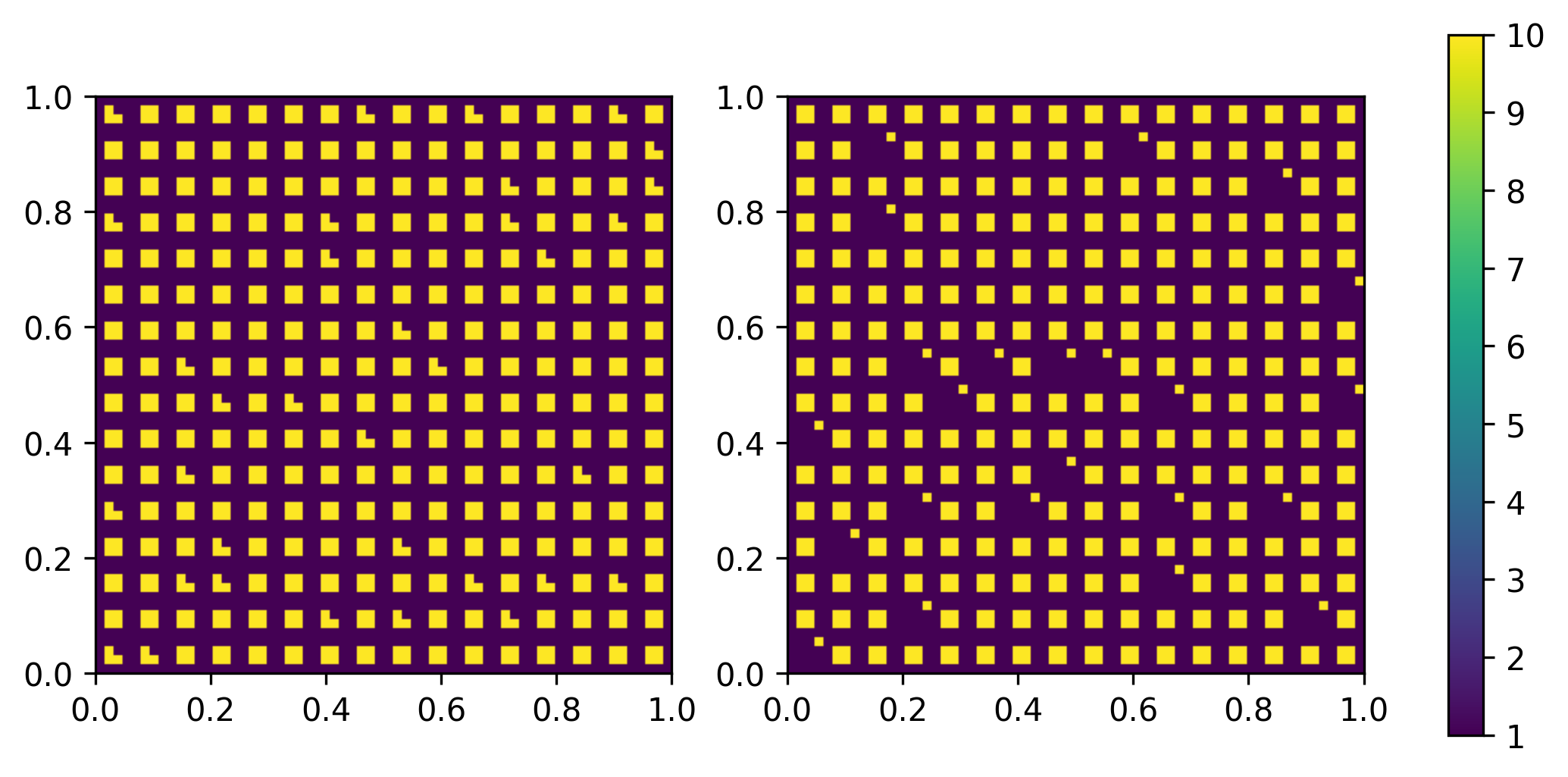}
\caption[0.6\textwidth]{The random defect L-shaped (left) and shifted (right) coefficients with $\alpha=1 $, $\beta=10, p=0.1,\, \varepsilon=2^{-5}$ and $h=2^{-9} $.}
\label{incl-models}
\end{figure}

\subsubsection*{L-shaped random defect model}
The L-shaped model illustrated in Figure~\ref{incl-models} (left) is obtained by choosing $\mathcal{Q}=[0.25,0.75]^2\backslash [0.5, 0.75]^2$ and modifying 
\begin{equation*}
    B_{\mathrm{per}}(y)= \begin{cases} \alpha-\beta, & y\in [0.5,0.75]^2,\\ 0, & \text{otherwise}.
\end{cases}
\end{equation*}
Observe that the defect geometry in this model breaks rotational symmetry but distributes the coefficient perturbation over a larger portion of the local patch compared to the compact square inclusions considered previously.

For all tested defect probabilities and contrasts, the PCG method once again converges for all three preconditioners. An interesting observation in Figure \ref{inclLshape-c} is that the ND-DD preconditioner exhibits smaller iteration counts for the L-shaped defect model than for the square inclusion model we observed in Figure \ref{incl-c}. This behavior can be explained by the spatial distribution of the defect. While the square inclusion introduces a compact and strongly localized perturbation of the coefficient, the L-shaped defect distributes the same contrast over a larger region of the patch. As a result, the deviation of the true local operator from the defect-free background operator is less pronounced, leading to a smaller spectral mismatch and improved convergence of ND-DD.
In contrast, the Direct-DD and OO-DD preconditioners are largely insensitive to this effect, as both methods explicitly account for the local defect structure. Consequently, their iteration counts remain comparable between the inclusion and L-shaped models. 

\begin{figure}[h]
\includegraphics[width=\textwidth]{ 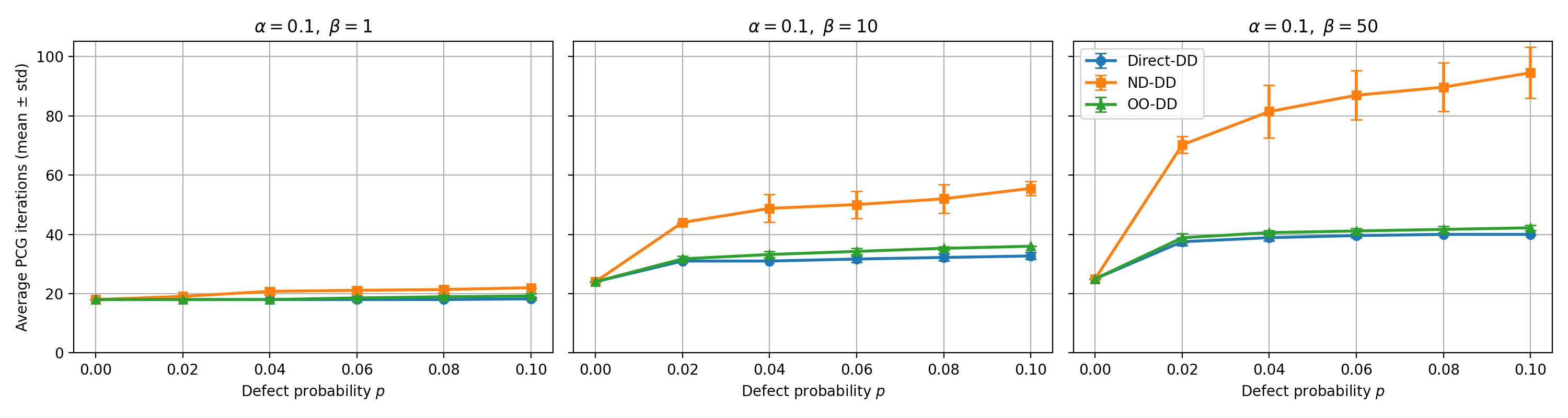}
\caption[0.6\textwidth]{L-shaped random defect model: the average PCG iteration counts against varying defect probabilities for defect contrasts $\beta/\alpha = 10, \, 100$ and $500$ (left to right) respectively.}
\label{inclLshape-c}
\end{figure}

\begin{figure}[H]
\includegraphics[width=\textwidth]{ 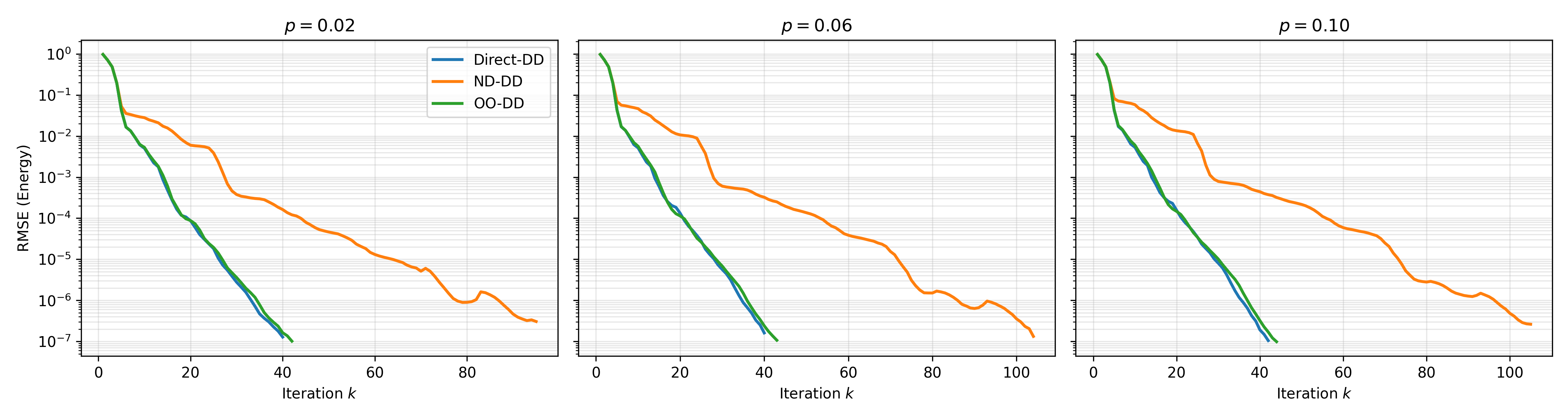}
\caption[0.6\textwidth]{L-shaped random defect model: the RMSE in the energy norm vs iteration count for defect probabilities $p= 0.02, \, 0.06$ and $0.10$ (left to right) respectively at $\beta/\alpha = 500$.}
\label{inclLshape-p}
\end{figure}

\subsubsection*{Shifted random defect model}

We finally consider the random defect shifted model shown in Figure~\ref{incl-models} (right), which represents the most challenging test case among the considered defect configurations. This model is obtained by considering $\mathcal{Q} = [0.75,1]^2$ together with
\begin{equation*}
    B_{\mathrm{per}}(y)= \begin{cases}
\alpha-\beta, & y\in [0.25,0.75]^2,\\
\beta-\alpha, & y\in [0.75,1]^2,\\
0, & \text{otherwise}.
\end{cases}
\end{equation*}
In contrast to the previous models, the defect is no longer centered within the $\varepsilon$-cell, although its fine-scale structure is preserved. This modification breaks the symmetry implicitly exploited by the ND--DD preconditioner.

Figure~\ref{inclshift-c} reports the average number of PCG iterations of the converged samples. For smaller defect contrasts e.g. $\beta/\alpha =10$, all methods converged even with a clearly increased number of iterations. At a contrast $\beta/\alpha =100$, the Direct-DD preconditioner remains robust for all tested defect probabilities, converging for all samples with only a mild increase in iteration counts. Similarly, the proposed OO-DD preconditioner converges for all samples across the full range of defect probabilities, with the average iteration count exhibiting an increase as the defect probability grows.
However, for higher contrast, in particular at $\beta/\alpha = 100$, the ND-DD preconditioner exhibits a severe loss of robustness. Already for small defect probabilities, convergence within the prescribed maximum of $200$ PCG iterations is observed only for a small fraction of the samples. More precisely, at defect probability $p=0.02$, only $10.67\%$ of all samples converged in $200$ iterations. For moderate defect probabilities, no convergence is achieved within this iteration limit. As a consequence, iteration counts for ND-DD are not reported beyond this regime.

The energy-norm RMSE convergence histories in  Figure~\ref{inclshift-p} further confirm these observations. While the OO-DD preconditioner shows a stable and monotone decay of the energy error, the error reduction in ND-DD preconditioner is much slower and the error does not fall below the prescribed tolerance within the maximum of $200$ iterations.

\begin{figure}[h]
\centering
\includegraphics[width=0.8\textwidth, height=5cm]{ 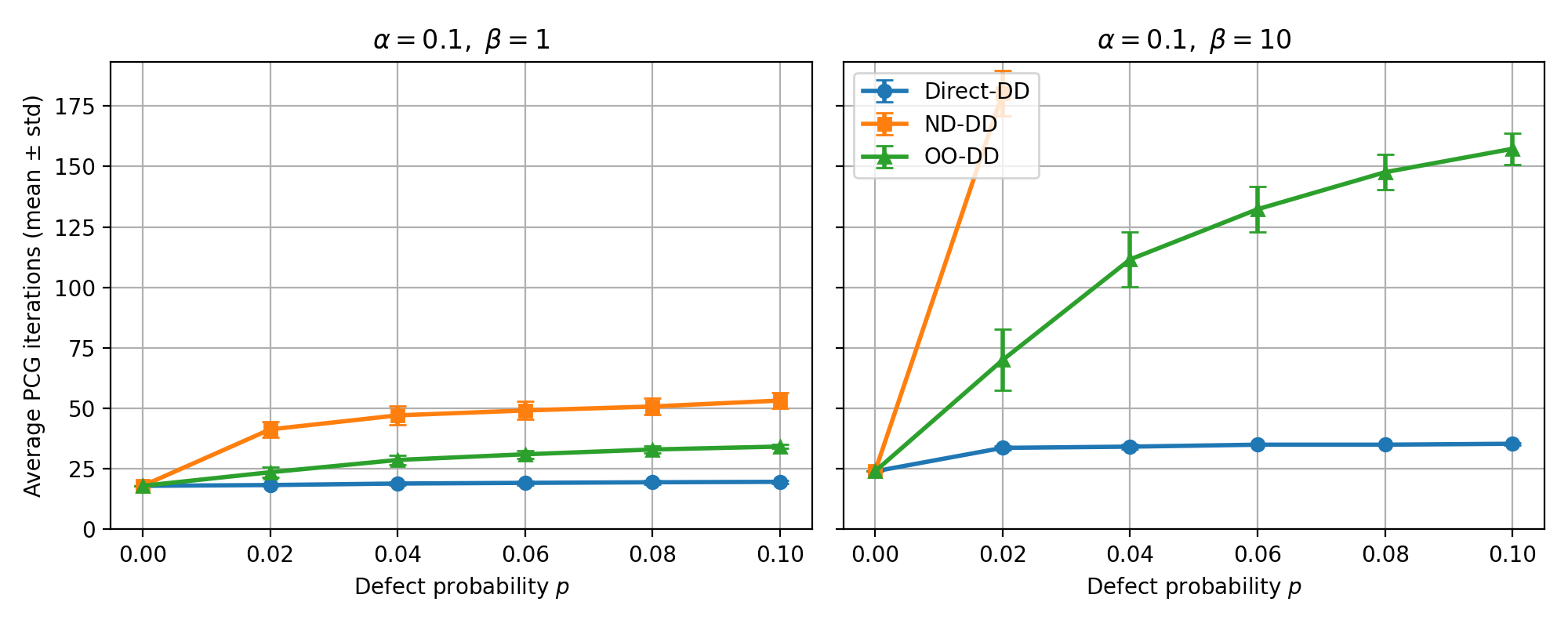}
\caption[0.6\textwidth]{Shifted random defect model: the average PCG iteration counts against varying defect probabilities for defect contrasts $\beta/\alpha = 10$ and $100$ (left to right) respectively.}
\label{inclshift-c}
\end{figure}

\begin{figure}[h]
\includegraphics[width=\textwidth]{ 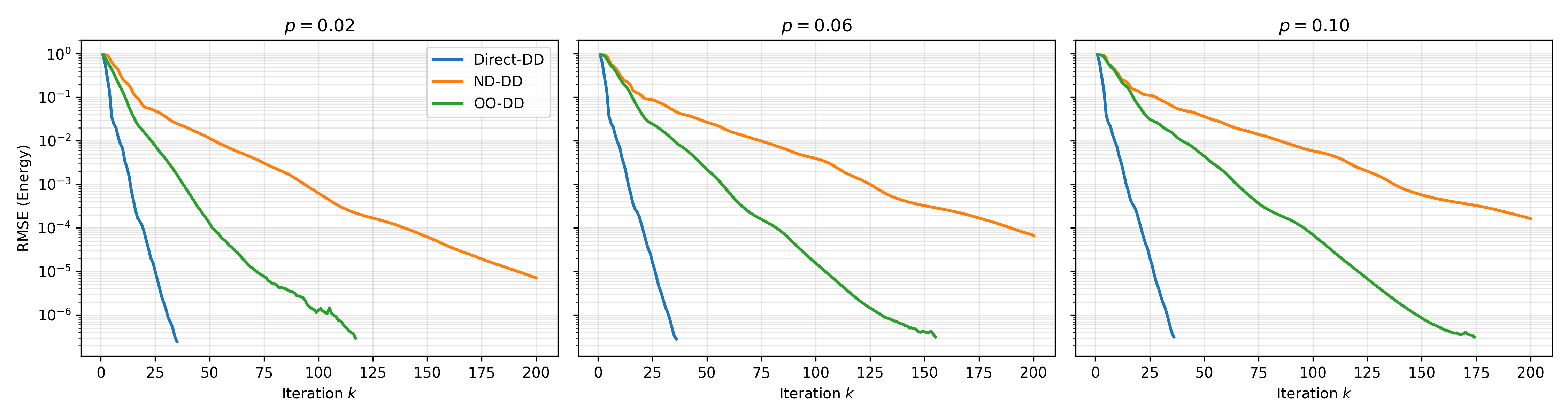}
\caption[0.6\textwidth]{Shifted random defect model: the RMSE in the energy norm vs iteration count for defect probabilities $p= 0.02, \, 0.06$ and $0.10$ (left to right) respectively at $\beta/\alpha = 100$.}
\label{inclshift-p}
\end{figure}

The numerical results across all defect models are fully consistent with the complexity analysis presented in Section~\ref{sec:5}. While the Direct-DD preconditioner yields the smallest iteration counts, its high per-sample cost due to fine-scale patch solves makes it computationally expensive in multi-sample settings. The ND-DD approach, despite its minimal setup cost, lacks robustness once defect symmetry/ centering is violated or considered at higher defect contrast, and therefore fails in more challenging regimes. In contrast, the OO-DD  preconditioner achieves robust convergence for all considered defect configurations with iteration counts close to those of Direct-DD, while incurring only inexpensive online recombination costs. These observations confirm that OO-DD provides the intended balance between robustness and computational efficiency predicted by the theoretical analysis.

\subsection{Experimental validation of the perturbation framework}
\label{sec:6.3}

In Section \ref{sec:4} we established that the offline-online approximation can be interpreted as a perturbation of the exact additive-Schwarz preconditioner. In particular, Theorem \ref{thm:oo-stability} shows that sufficiently small perturbations preserve positive definiteness of the offline-online preconditioner, while Proposition \ref{prop:eta} provides a computable upper bound for the perturbation constant. In this section we validate these theoretical findings numerically. Rather than explicitly evaluating the quantities $E_i$ from Proposition \ref{prop:eta}, which would require solving additional local eigenvalue problems for every patch and coefficient realization, we consider two complementary numerical indicators. Specifically, we first verify the preservation of positive definiteness of the offline-online preconditioner and then examine the magnitude of the offline-online perturbation through a practical operator based measure. 

First, to verify the positive definiteness, we computed the smallest eigenvalue of $\widetilde{B}(A)$. For over $150$ Monte Carlo realizations of Experiment \ref{ex:1} at $\beta/\alpha =10$ for all considered probabilities, the smallest observed eigenvalue remained strictly positive. These exact values can be found in the repository \cite{dilini_kolombage_2026_21769877}. This can similarly be verified easily also for all other experiments as well.

Second, we assess the magnitude of the offline-online perturbation through the relative deviation 
\begin{equation*}
\frac{\| (\overline {\mathcal B}-\mathcal B) r \|_2}{\| \mathcal B r \|_2},
\end{equation*}
for random test vectors $r$. Here, $\overline B$ denotes either $\widetilde {\mathcal B}$ or $\mathcal B(A_0)$ (ND-DD). The reported values are obtained  by computing the RMSE of this relative deviation across all samples. Although this quantity is not an estimator for the perturbation constant $\eta$, it provides complementary empirical evidence of how accurately the offline-online approximation reproduces the Direct-DD preconditioner and allows a direct comparison with the ND-DD preconditioner. This, in turn, helps explain the robustness of the proposed method observed in the previous experiments. Figure \ref{fig:Bdiff-incl} shows the resulting relative deviations for the random inclusion model for varying defect probabilities. For both methods mean relative deviation increase as $p$ increases, reflecting the growing deviation of the coefficient from its approximation as more defects come into play. However, the OO-DD preconditioner $\widetilde {\mathcal B}$ remains consistently much closer to the Direct-DD preconditioner $\mathcal B$ than the ND-DD preconditioner $\mathcal B(A_0)$ across the entire range of the considered defect probabilities. This behavior is consistent with the perturbation framework developed in Section 4 and explains why the convergence of OO-DD remains much closer to that of Direct-DD, whereas ND-DD deteriorates more rapidly.
\begin{figure}[h]
\centering
\includegraphics[width=0.45\textwidth, height=4.6cm]{ 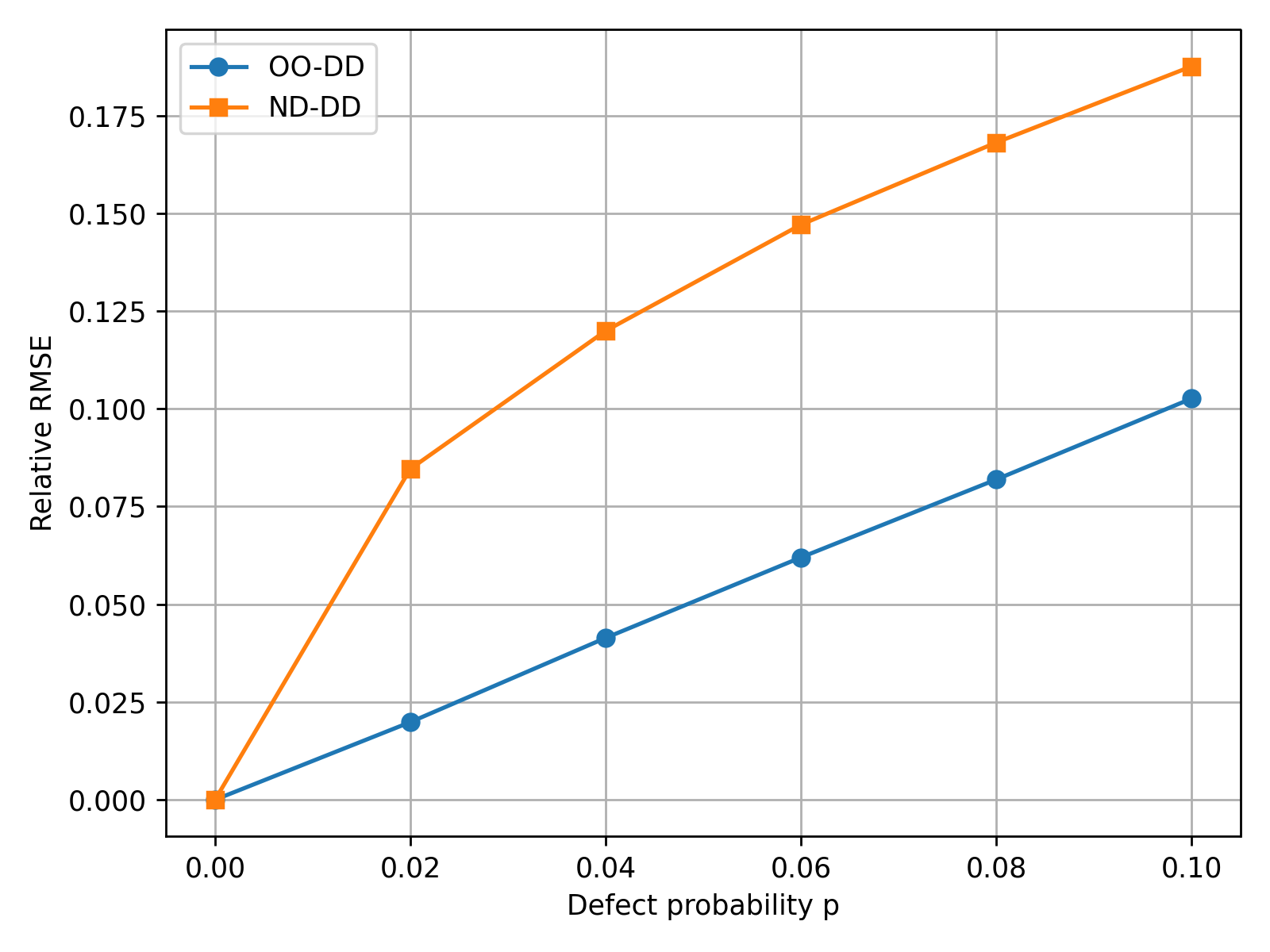}
\caption[0.6\textwidth]{The operator behavior for varying defect probabilities in random erasure model with contrast $\beta/\alpha = 500$.}
\label{fig:Bdiff-incl}
\end{figure}

\section{Conclusion}
\label{sec:7}
In this work, we proposed and analyzed a two-level domain decomposition preconditioner for diffusion problems with localized random defects. The method is based on a subspace decomposition framework in which expensive fine-scale patch solves are computed in an offline stage and reused in the online stage through linear combinations of reference operators corresponding to single-defect configurations.

From a theoretical perspective, we derived perturbation-based bounds that quantify the effect of replacing exact local operators by offline-online approximations. These results indicate that the spectral properties of the preconditioned system degrade only moderately when the defect structure is well captured by the chosen reference space. A detailed run-time complexity analysis further demonstrated that, beyond a small break-even number of samples, the proposed method is expected to be significantly more efficient than direct domain decomposition approaches.

Extensive numerical experiments in two dimensions confirmed these theoretical findings. Across a range of defect probabilities, contrasts, and geometries, the OO-DD preconditioner consistently achieved convergence behavior close to that of the exact Direct-DD method while substantially reducing the per-sample computational cost. In particular, for shifted defect configurations that break defect centering, the ND-DD preconditioner was observed to lose convergence of the PCG solver, whereas the OO-DD method remained fully convergent. Overall, the proposed offline-online strategy offers a robust and efficient preconditioning approach for multi-query diffusion problems with localized defects. Future work includes adaptive selection of reference defect configurations and extensions of the reference space to include multiple and overlapping defect configurations.

\section*{Acknowledgments}
The authors would like to thank the Hausdorff Institute for Mathematics (HIM), Bonn, for its hospitality, where part of this work was completed during the trimester program ``Computational multifidelity, multilevel, and multiscale methods''.

\section*{Funding}
This research is funded by the Deutsche Forschungsgemeinschaft (DFG, German Research Foundation) under project number 496556642. BV additionally acknowledges support from the Deutsche Forschungsgemeinschaft (DFG, German Research Foundation) under Germany's Excellence Strategy – EXC-2047/1 – 390685813. AM acknowledges support from the Swedish Research Council, project number 2023-03258 VR.

\section*{Data availability statement}
The research data/code associated with this article are available in https://github.com/DKolombage/random-perturbations-preconditioning under the reference https://doi.org/10.5281/zenodo.21769877.

\bibliographystyle{abbrv}
\bibliography{references}
\end{document}